\documentclass[11pt, reqno]{amsart}

\usepackage[utf8]{inputenc}
\usepackage{fullpage}
\usepackage{amsthm, amsfonts, amsmath, amssymb, amsrefs, bbm}
\usepackage{tikz, relsize, amsaddr, mathrsfs, graphicx, verbatim, paralist, upgreek, caption}
\usepackage[breaklinks, pdfstartview=FitH]{hyperref}
\usepackage[mathscr]{euscript}
\usepackage{enumitem,kantlipsum}

\hypersetup{ocgcolorlinks=true, allcolors=testc}
\hypersetup{
    colorlinks = true,
    citecolor = black
}
\hypersetup{linkcolor = black}
\hypersetup{urlcolor = black}

\newtheorem{theorem}{Theorem}[section]
\newtheorem{proposition}{Proposition}[section]
\newtheorem*{corollary*}{Corollary}
\newtheorem*{remark*}{Remark}
\newtheorem{lemma}{Lemma}[section]
\newtheorem{definition}{Definition}[section]

\begin{document}

\title{Randomized Systematic Scan dynamics on the Mean-field Ising Model}
\author{Sanghak Jeon}
\address{Mathematics Department, Princeton University, Princeton, NJ 08544-1000, USA}
\thanks{{\it E-mail addresses: } \href{mailto:sanghakj@princeton.edu}{\nolinkurl{sanghakj@princeton.edu}}}.

\begin{abstract}
    We study the mixing time of systematic scan Glauber dynamics Ising model on the complete graph. On the complete graph $K_n$, at each time, $k \leq n$ vertices are chosen uniformly random and are updated one by one according to the uniformly randomly chosen permutations over the $k$ vertices. We show that if $k = o(n^{1/3})$, the high temperature regime $\beta < 1$ exhibits cutoff phenomena. For critical temperature regime $\beta = 1$, We prove that the mixing time is of order $n^{3/2}k^{-1}$. For $\beta > 1$, we prove the mixing time is of order $nk^{-1}\log n$ under the restricted dynamics.
\end{abstract}

\maketitle

{\footnotesize
    \noindent {\em 2023 Mathematics Subject Classification.} Primary: 60J10; Secondary: 60C05, 60G42.
    
    \noindent {\em Keywords and phrases.} Markov chain, mixing time, Ising model, systematic scan, cutoff phenomenon.
}

\section{Introduction}
    The Ising model has risen in Statistical Physics to explain the ferromagnetism, and its mixing time has been extensively studied, yet there are only few cutoff results(\cite{DLY09}, \cite{DLP09}, \cite{LLP10}, \cite{LS12}, \cite{LS14}). In this paper we propose the Ising model under systematic scan, suggest the mixing time order for all regimes, and prove the existence of the total variation cutoff phenomenon for a high temperature regime.

    One of the common ways to evolve the Ising model is the Glauber dynamics: choosing a vertex $v$ uniformly randomly from a given graph and update its spin according to the spins of the vertices connected to $v$. From a different point of view, rather than choosing a vertex uniformly randomly, we can consider a model with which is updated in a deterministic way. This is so-called systematic scan, and a folklore belief on several problems is that models with a random update dynamics and systematic scan share a similar have similar metastability. To be more precise, for the Ising model on a finite graph $G = (V, \mathcal{E})$, fix $k \leq |V|$ and uniformly randomly choose $k$ vertices among $V$, uniformly randomly choose a permutation over those $k$ vertices and update one by one along the permutation.
    
    The question that comes along with a given Markov chain is how fast it converges to its stationary distribution. Total variation distance on Markov chain measures the difference between the worst probability distribution and the stationary measure:
    \begin{equation*}
        d(t) := \max_{x \in X} \left\| P^t_x - \pi \right\|_{TV} = \frac{1}{2}\max_{x \in X}\sum_{y \in X}|P^t(x, y) - \pi(y)|,
    \end{equation*}
    and the mixing time associated with the distance
    \begin{equation*}
         t_\textrm{mix}(\epsilon) := \min \{t : d(t) \leq \epsilon\}
    \end{equation*}
    quantifies the convergence rate. Often $t_\textrm{mix}(1/4)$ is written as $t_\textrm{mix}$. In addition, with a sequence of graphs $G_n$ we can investigate their asymptotic behavior around the mixing time. For a sequence of Markov chains $\{X_n\}$ with the distance $d_n(t)$, we say the chain exhibits \textit{cutoff} at $\{t_n\}$ with window size $\{w_n\}$ if $w_n = o(t_n)$ and
    \begin{equation*}
        \begin{split}
            \lim_{\gamma \to \infty} \liminf_{n\to\infty} d_n(t_n - \gamma w_n) &= 1, \\
            \lim_{\gamma \to \infty} \limsup_{n\to\infty} d_n(t_n + \gamma w_n) &= 0.
        \end{split}
    \end{equation*}
    There is an equivalent statement for cutoff phenomena, but it does not provide any information about the window size. See \cite{LP10} chapter 18 for more details on cutoff phenomena.
    
    To this end we can pose a problem on systematic scan Ising model: \textit{Do Ising models with Glauber dynamics and with systematic scan dynamics have the same order of the mixing time, and cutoff phenomena?} One of the main theorems of this paper is that the complete graph Ising model equipped with systematic scan dynamics has a total variation cutoff in the high temperature regime, $\beta < 1$.
    
    \begin{theorem} \label{thm: main1}
        Suppose $\beta < 1$. If $k = o(n^{1/3})$, then the model exhibits cutoff at $t_n = [2k(1-\beta)]^{-1}n \log n$ with window size $w_n = n/k$.
    \end{theorem}

    For the complete graph Ising model, $\beta = 1$ is known for not exhibiting a cutoff. Hence we only propose both the upper and lower bound of the mixing time under systematic scan:

    \begin{theorem} \label{thm: main2}
        Suppose $\beta = 1$. If $k = o(n^{1/4})$, then there exists constants $c_1, c_2 > 0$ such that $c_1 n^{3/2}k^{-1} \leq t_\textrm{mix}(1/4) \leq c_2 n^{3/2}k^{-1}$.
    \end{theorem}

    For the last regime, $\beta > 1$, the mixing time is known to be exponential in $n$. \cite{LLP10} suggested a restricted version of Glauber dynamics and they showed the mixing time is $O(n\log n)$. The restricted dynamics can be also applied to systematic scan case, and we prove the mixing time order is $nk^{-1}\log n$ in this model.
    
    \begin{theorem} \label{thm: main3}
        Suppose $\beta > 1$. If $k = o(n^{1/2})$, then there exists constants $c_1, c_2 > 0$ such that $c_1 nk^{-1} \log n \leq t_\textrm{mix}(1/4) \leq c_2 nk^{-1}\log n$, where $t_\textrm{mix}(1/4)$ is the mixing time under the restricted dynamics.
    \end{theorem}

    There are two remarks on this dynamics compared to the block dynamics: Randomized systematic scan dynamics updates $k$ vertices one by one while block dynamics updates $k$ vertices at once. Unlike the block dynamics, randomized systematic scan dynamics generates a small amount of error during $k$ numbers of single-site updates. This error accumulation makes magnetization analysis more difficult. Secondly, if $k$ is large, computer modeling on randomized scan dynamics is more efficient than block dynamics, as one update with randomized scan Glauber dynamics can be divided into $k$ single-site updates, which requires much less memory storage.

    \cite{DGJ08} showed the mixing time of the Curie-Weiss model in high temperature regime is of order $n\log n$ under systematic scan. We conjecture the highest order term for the mixing time for $k = n$ case is $[2(1-\beta)]^{-1}\log n$, which is consistent to Theorem \ref{thm: main1}. We expect there must be a cutoff under both systematic scan and the randomized systematic scan. All three abovementioned Theorems \ref{thm: main1}, \ref{thm: main2}, and \ref{thm: main3} are proven by showing both upper bound and lower bound inequalities. All the upper bound are mostly derived from coupling argument, however as $k$ numbers of vertices are updated at each time, several calculations are required to finalize the proofs. The lower bound arguments are heavily relying on the estimates on magnetizations. Section 3, 4, and 5 contains Theorem $\ref{thm: main1}$, \ref{thm: main2}, and \ref{thm: main3} and their proofs respectively. Section 6 briefly discuss the case $k=n$, which is the model without coupon collecting properties. One of the proofs of the lemma which is crucial for Theorem \ref{thm: main3} is moved to the Appendix, Section 7.
     
    \smallskip    
    
    \subsection{Background and previous researches}
    
    The mixing time and cutoff results for the Ising model on the complete graph, also known as Curie-Weiss model, were established on \cite{LLP10}, \cite{DLY09} and \cite{DLP09}, showing that there is a cutoff only in the high temperature regime. For the lattice, Lubetzky and Sly \cite{LS12} showed that there is a cutoff on $(\mathbb{Z} / n\mathbb{Z})^d$ with strong spatial mixing, which always holds when $d = 2$. There is a result on the Ising model on a regular tree also, yet people are working on 3 or higher dimensional lattice. \cite{DP11} suggested a lower bound argument for the mixing time of the Ising model over an arbitrary graph, and \cite{Nes19} suggested another lower bound argument, in terms of the separation distance mixing time.

    Systematic scan naturally emerged from Markov chains, particularly from the card shuffling problems. Since systematic scan excludes coupon-collecting properties from the dynamics, one might expect systematic scan to accelerate the mixing. However, there are only few results on systematic scan dynamics, and several results showed that the random updates chain and systematic scan chain have the same mixing time order. For example, \cite{MPS04} and \cite{SZ07} are systematic scan version of \cite{DS81}, which showed the mixing time of two schemes are in the same order. There are some techniques for systematic scan \cite{SZ07}, \cite{PW13}, \cite{DR04}; however systematic scan dynamics still poses challenges to be analyzed. \cite{bcsv17} is one of few results on the Ising model and systematic scan simultaneously. \cite{bcsv17} studied the mixing time of the Ising model over $d-$dimensional lattice, under the systematic scan.
    
    Particularly for the Ising model over the complete graph, It has been well known that the model's mixing time is of order $n\log n$ under the Glauber dynamics. Furthermore \cite{DGJ08} proved that the order of the mixing time under systematic scan Glauber dynamics is the same as the original one, by utilizing the Dobrushin-Shlosman condition. Nonetheless the existence of cutoff remained uncovered since the lack of symmetry prevents the coupling argument and vertex rematching - the fundamental idea of \cite{LLP10}.

\section{Preliminaries}
    In this section we present some definitions and several properties, which are used for the remainder of the paper. The Ising model on a given finite graph $G = (V,\mathcal{E})$ with a parameter $\beta > 0$ is a probability distribution over the state space $\Omega := \{+1, -1\}^V$ with the probability of $\sigma \in \Omega$ given by the Gibbs measure
    \begin{equation*}
        \pi(\sigma) = \frac{1}{Z(\beta)}\exp\left(\beta \sum_{(vw) \in \mathcal{E}} J(v, w)\sigma(v)\sigma(w)\right),
    \end{equation*}
    where $Z(\beta)$ is a normalizing constant. We assume there are no external fields. $\beta$ is often interpreted as an inverse-temperature, is chosen to be non-negative for ferromagnetism.

    One of the common ways to evolve the Ising model is the Glauber dynamics, which is the following: from configuration $\sigma$, a vertex $v \in V$ is uniformly chosen, and a new configuration is selected from the set $\{\eta \in \Omega: \eta(w) = \sigma(w), w \neq v\}$. The new configuration and $\sigma$ agree on all vertices except at $v$, and at $v$ the new spin is $\pm1$ with probability
    \begin{equation} \label{eq: 1-1}
        p(\sigma;v) = \frac{e^{\pm\beta S^\nu(\sigma)}}{e^{\beta S^\nu(\sigma)} + e^{-\beta S^\nu(\sigma)}},
    \end{equation}
    where $S^\nu(\sigma) := (1/n)\sum_{w : (wv) \in \mathcal{E}} \sigma(w)$. Therefore, the new spin at $v$ only depends on the current spins of the neighboring vertices of $v$. Pick an element from $\Omega$ and evolve with the Glauber dynamics generates a discrete-time Markov chain.

    From now on, the Ising model under the randomized systematic scan dynamics will be denoted by $\{X_t\}^\infty_{t=0}$. We use $\mathbb{P}_\sigma$ and $\mathbb{E}_\sigma$ to denote the probability measure and associated expectations given $X_0 = \sigma$. We define a coupling of the dynamics as a process $(X_t, \widetilde{X}_t)_{t\geq 0}$, where each $\{X_t\}$ and $\{\widetilde{X}_t\}$ are versions of the dynamics. Similarly we write $\mathbb{P}_{\sigma, \widetilde{\sigma}}$ and $\mathbb{E}_{\sigma, \widetilde{\sigma}}$ for the probability measure and associated expectation respectively, starting from $\sigma$ and $\widetilde{\sigma}$.
    
    For a configuration $\sigma \in \{-1, 1\}^V$, denote $\sigma(v)$ as a spin on the vertex $v$. As we mainly focus on complete graphs, sometimes we denote $\sigma(i)$ for $i \in \{1, ..., n\}$ as a spin on the $i$-th vertex. Define \textit{magnetization} of $\sigma$ as the average of all spins of $\sigma$:
    \begin{equation*}
        S(\sigma) := \frac{1}{n}\sum_{v\in V}{\sigma(v)}.
    \end{equation*}
    We are going to denote the magnetization of $X_t$ as $S_t := S(X_t)$, for simplicity's sake. Due to the symmetry of $K_n$, defining those functions are helpful to abbreviate equations:
    \begin{equation} \label{eq: 2-1}
        \begin{split}
            p_{+}(x) &:= \frac{1+\tanh{\beta x}}{2} \\
            p_{-}(x) &:= \frac{1-\tanh{\beta x}}{2} = 1 - p_{+}(x)
        \end{split}
    \end{equation}
    Then the probability $p(\sigma; v)$ from Equation \eqref{eq: 1-1} equals $p_\pm\left(S(\sigma) - \sigma(v)/n\right)$.

    We uniformly randomly choose $k \leq n$ vertices among $n$ vertices and pick a permutation uniformly randomly on these selected vertices. We update $k$ vertices in that order and repeat the entire process. This randomized systematic scan updates $k$ vertices at each time $t$, we may try to split it into $k$ numbers of single-site updates; to elaborate, we can consider $(k+1)$ states $Y_0 = X_0, Y_1, ..., Y_{k-1}, Y_k = X_1$ such that $Y_{i+1}$ is derived from the single-site update from $Y_i$. As $i$ number of vertices of $Y_i$ cannot be updated, $\{Y_i\}$ is not a Markov chain. However, since $S(Y_{i+1}) - S(Y_i) \in \{\pm2/n, 0\}$, we can apply several birth-and-death chains results. Although those results need modifications as $S(Y_i)$ is not a Markov chain, considering $\{Y_i\}$ still useful for coupling arguments and hitting time estimates. We denote these $\{Y_i\}$ as \textit{intermediate states}. Sometimes we consider the extended version of intermediate states $\{Y_i\}_{0\leq i \leq kT}$ such that $Y_{ki} = X_i$ for all $i \leq T$. For each intermediate state $Y_i$, $i$(mod $k$) vertices cannot be selected for the update. Call these vertices as \textit{not available vertices} and denote as $\mathcal{N}_i$.
    
    Define the \textit{grand coupling} $\{X_{\sigma, t}\}_{\sigma \in \Omega, t \geq 0}$ from all states of $\Omega$ as the following: for each $\sigma \in \Omega$, the coordinate process $(X_{\sigma, t})_{t \geq 0}$ is the Glauber dynamics that starts from $\sigma$. For each time $t \geq 0$, let $\{Y_{\sigma, i}\}_{0\leq i\leq k}$ be all intermediates such that $Y_{\sigma, 0} = X_{\sigma, t}$. Randomly choose the ordered set of vertices $\{v_1, ..., v_k\}$ to update. $Y_{\sigma, i}$ denotes the configuration started from $X_{\sigma, t}$ with updates on $\{v_1, ..., v_i\}$. Define $U_1, U_2, ..., U_k$ as copies of random variables which are uniform on $[0, 1]$ and let $U_i$ determine the spin $Y_{\sigma, i}(v_i)$ for all $\sigma \in \Omega$ by
    \begin{equation*}
        Y_{\sigma, i}(v_i) =
        \begin{cases}
            +1 &\textrm{if} \quad 0 \leq U \leq p_{+}(S(Y_{\sigma, i-1}) - n^{-1}Y_{\sigma, i-1}(v_i))\\
            -1 &\textrm{if} \quad p_{+}(S(Y_{\sigma, i-1}) - n^{-1}Y_{\sigma, i-1}(v_i)) < U \leq 1,
        \end{cases}
    \end{equation*}
    and $Y_{\sigma, i}(v) = Y_{\sigma, i-1}(v)$ for all $v \neq v_i$. Finally accept $X_{\sigma, t+1} = Y_{\sigma, k}$. The construction ensures that $X_{\sigma, t+1}(w) = X_{\sigma, t}(w)$ for all vertices but $\{v_1, ..., v_k\}$. For any two given configurations $\sigma$ and $\widetilde{\sigma}$, define \textit{monotone coupling} $(X_{\sigma, t}, X_{\widetilde{\sigma}, t})$ as the projection of the grand coupling with starting configurations $\sigma$ and $\widetilde{\sigma}$. 
    
    For two spin configurations $\sigma$ and $\widetilde{\sigma}$, define their Hamming distance as the number of disagreeing vertices:
    \begin{equation*}
        \textrm{dist}(\sigma, \widetilde{\sigma}) := \frac{1}{2}\sum^{n}_{j=1}|\sigma(j) - \widetilde{\sigma}(j)|.
    \end{equation*}
    Then Hamming distance contracts under monotone coupling in high temperature regime.
    
    \begin{proposition} \label{prop: hamming distance contraction}
        Under the randomized systematic scan dynamics, the monotone coupling $(X_t, \widetilde{X}_t)$ with $(X_0, \widetilde{X}_0) = (\sigma, \widetilde{\sigma})$ satisfies
        \begin{equation} \label{eq: 2-2}
            \mathbb{E}_{\sigma, \widetilde{\sigma}}\big[ \normalfont\textrm{dist}(X_t, \widetilde{X}_t)\big] \leq \frac{1}{\beta}\Big[1 + (\beta-1)\Big(1+\frac{\beta}{n}\Big)^k\Big]^t \normalfont\textrm{dist}(\sigma, \widetilde{\sigma}).
        \end{equation}
    \end{proposition}
    
    \begin{remark*}
        Proposition \ref{prop: hamming distance contraction} remains true for any $k \leq n$. If $\beta < 1$,
        \begin{equation*}
            \rho := \frac{1}{\beta}\Big[1 + (\beta-1)\Big(1+\frac{\beta}{n}\Big)^k\Big] \leq 1 - \frac{k(1-\beta)}{n} < 1
        \end{equation*}
    holds, thus Hamming distance tends to decrease in the high temperature regime.
    \end{remark*}
    
    \begin{proof}
        We consider the special case when dist$(\sigma, \widetilde{\sigma}) = 1$ and $t = 1$. Let $I$ be a vertex that has two different spins: $1 = \sigma(I) \neq \widetilde{\sigma}(I) = -1$. With probability $\frac{n-k}{n}$ the vertex $I$ is not selected to be updated; let $\mathcal{P}$ be the ordered $k$ vertices set to be updated. From $X_0$ to $X_1$, if $I \notin \mathcal{P}$, in terms of intermediate states $\{Y_i\}_{0\leq i \leq k}$ and $\{\widetilde{Y}_i\}_{0 \leq i \leq k}$,
        \begin{equation} \label{eq: 2-3}
            \begin{split}
                &\mathbb{E}_{\sigma, \widetilde{\sigma}} \big[\textrm{dist}(Y_{i+1}, \widetilde{Y}_{i+1}) | \textrm{dist}(Y_i, \widetilde{Y}_i)\big] \\
                &\qquad = \textrm{dist}(Y_i, \widetilde{Y}_i) + \big|p_{+}(S(Y_i) - n^{-1}Y_i(v)) - p_{+}(S(\widetilde{Y}_i) - n^{-1}\widetilde{Y}_i(v))\big| \\
                &\qquad \leq \textrm{dist}(Y_i, \widetilde{Y}_i) + \tanh\Big(\beta\frac{|S(Y_i) - S(\widetilde{Y}_i)|}{2}\Big) \\
                &\qquad \leq \Big(1 + \frac{\beta}{n}\Big)\textrm{dist}(Y_i, \widetilde{Y}_i),
            \end{split}
        \end{equation}
        for some $v\neq I$. Otherwise, suppose $I$ is chosen to be $j$-th updated vertex. Equation \eqref{eq: 2-3} holds for every indices $i \neq j-1$, but for $j - 1$,
        \begin{equation*}
            \begin{split}
                &\mathbb{E}_{\sigma, \widetilde{\sigma}} \big[\textrm{dist}(Y_{j}, \widetilde{Y}_{j}) | \textrm{dist}(Y_{j-1}, \widetilde{Y}_{j-1}) \big] \\
                &\qquad = \textrm{dist}(Y_{j-1}, \widetilde{Y}_{j-1}) - \left(1 - \big|p_{+}(S(Y_{j-1}) - n^{-1}Y_{j-1}(v)) - p_{+}(S(\widetilde{Y}_{j-1}) - n^{-1}\widetilde{Y}_{j-1}(v))\big|\right) \\
                &\qquad \leq \textrm{dist}(Y_{j-1}, \widetilde{Y}_{j-1}) + \tanh\Big(\beta\frac{S(Y_{j-1}) - S(\widetilde{Y}_{j-1}) - 2/n}{2}\Big) - 1 \\
                &\qquad \leq \Big(1 + \frac{\beta}{n}\Big)\big(\textrm{dist}(Y_{j-1}, \widetilde{Y}_{j-1})-1\big),
            \end{split}
        \end{equation*}
        due to the fact that $S(Y_i) \geq S(\widetilde{Y}_i) + 2/n$ for any $i \leq j-1$. Other than $i = j-1$ we can apply Equation \eqref{eq: 2-3} hence we have
        \begin{equation} \label{eq: 2-4}
            \mathbb{E}_{\sigma, \widetilde{\sigma}} [\textrm{dist}(X_1, \widetilde{X}_1)|\textrm{dist}(X_0, \widetilde{X}_0)] \leq \left(1 + \frac{\beta}{n}\right)^k - \left(1 + \frac{\beta}{n}\right)^j
        \end{equation}
        for $I \in \mathcal{P}$. Combining Equation \eqref{eq: 2-3} and Equation \eqref{eq: 2-4},
        \begin{equation} \label{eq: 2-5}
            \begin{split}
                &\mathbb{E}_{\sigma, \widetilde{\sigma}} \big[\textrm{dist}(X_1, \widetilde{X}_1)|\textrm{dist}(X_0, \widetilde{X}_0)\big] \\
                &\qquad \leq \frac{1}{n}\sum_{i=0}^{k-1}\Big[\Big(1+\frac{\beta}{n}\Big)^k - \Big(1 + \frac{\beta}{n}\Big)^{i+1}\Big] + \frac{n-k}{n}\Big(1 + \frac{\beta}{n}\Big)^k \\
                &\qquad = \frac{1}{\beta}\Big\{1 + \frac{\beta}{n}\Big[1-\Big(1 + \frac{\beta}{n}\Big)^k\Big] + (\beta-1)\Big(1 + \frac{\beta}{n}\Big)^k \Big\} \\
                &\qquad \leq \frac{1}{\beta}\Big[1 + (\beta-1)\Big(1 + \frac{\beta}{n}\Big)^k \Big].
            \end{split}
        \end{equation}
        To finish the proof, for general $\sigma, \widetilde{\sigma}$, we can choose dist$(\sigma, \widetilde{\sigma}) +1$ numbers of states so that they form a sequence from $\sigma$ to $\widetilde{\sigma}$ that the Hamming distances between any consecutive states are 1. Equation \eqref{eq: 2-5} and trigonometric inequality finishes this case, and recursive application gives Equation \eqref{eq: 2-2}.
    \end{proof}
    
    For simplicity of notation, we write $S_t := S(X_t)$ from now on. As $|S_t - \widetilde{S}_t| \leq (2/n)\textrm{dist}(X_t, \widetilde{X}_t)$ at any time $t$, we can transform Proposition \ref{prop: hamming distance contraction} to an equation in terms of magnetization.
    
    \begin{proposition} \label{prop: magnetization contraction}
        Suppose $\beta < 1$. For any two configurations $\sigma$ and $\widetilde{\sigma}$, under the monotone coupling,
        \begin{equation} \label{eq: 2-6}
            \mathbb{E}_{\sigma, \widetilde{\sigma}}\left[ |S_t - \widetilde{S}_t| \right] \leq \frac{2}{n}\rho^t \normalfont\textrm{dist}(\sigma, \widetilde{\sigma}) \leq 2\rho^t,
        \end{equation}
        where $\rho = 1 - k(1-\beta)/n < 1$.
    \end{proposition}
    
    The constant $\rho$ comes from Proposition \ref{prop: hamming distance contraction}. Furthermore, for any two configurations $\sigma$ and $\widetilde{\sigma}$, with the same constant $\rho < 1$ we can prove
    \begin{equation} \label{eq: 2-7}
        |\mathbb{E}_\sigma[S_1] - \mathbb{E}_{\widetilde{\sigma}}[S_1]| \leq \rho|S_0 - \widetilde{S}_0|
    \end{equation}
    in an analogous way.
    
    Often magnetization becomes supermartingale, so the lemma below is utilized throughout the paper. This lemma is suggested from \cite{LP10} chapter 18 with its proof so we omit the proof here.
    
    \begin{lemma} \label{lem: supermartingale lemma}
        Let $(W_t)_{t\geq0}$ be a non-negative supermartingale and $\tau$ be a stopping time such that
        \begin{enumerate}[leftmargin=*]
            \item $W_0 = k$
            \item $W_{t+1} - W_t \leq B$
            \item  $\normalfont \textrm{Var}(W_{t+1}|\mathcal{F}_t) > \sigma^2 > 0 \textrm{ on the event } \tau > t$.
        \end{enumerate}
        If $u > 4B^2/(3\sigma^2)$, then
        \begin{equation*}
            \mathbb{P}_k(\tau > u) \leq \frac{4k}{\sigma\sqrt{u}}.
        \end{equation*}
    \end{lemma}

    Lemma \ref{lem: supermartingale lemma} requires a lower bound on a variance. Each step of the original Glauber dynamics generates a variance of magnetization of order $1/n^2$. Since the randomized systematic scan consists of $k$ numbers of single-site update, we can guess the variance should be of order $k/n^2$.
    
    \begin{lemma} \label{lem: variance estimate no.1}
        Suppose $k = o(n)$. For any randomized scan dynamics chain $X_t$ with arbitrary starting state $\sigma$, $\normalfont \textrm{Var}_\sigma[S_1]$ is of order $k / n^2$.
    \end{lemma}
    
    \begin{proof}
        Define $v_i := \max\textrm{Var}_\sigma[S(Y_i)]$ for $t = 0, 1, ..., (k-1)$, where $Y_i$ is an intermediate state and the maximum is taken over all $\sigma \in \Omega$ and all possible vertex update permutations. We have $v_1 = \Omega(n^{-2})$. Let $\mathcal{R} = S(Y_{i+1}) - S(Y_{i}) \in \{-2/n, 0, 2/n\}$. In regards to the variation identity,
        \begin{equation*}
            \begin{split}
                v_{i+1} &= \max_{\sigma \in \Omega} \textrm{Var}_\sigma[S(Y_{i+1})] \\
                &= \max_{\sigma\in\Omega} \Big\{\mathbb{E}\big[\textrm{Var}_\sigma[S(Y_{i+1}) | \mathcal{R}]\big] + \textrm{Var}\big[\mathbb{E}_\sigma[S(Y_{i+1})|\mathcal{R}]\big] \Big\}
            \end{split}
        \end{equation*}
        The first term inside the parenthesis satisfies
        \begin{equation*}
            \mathbb{E}\big[\textrm{Var}_\sigma[S(Y_{i+1}) |\mathcal{R}]\big] = \mathbb{E}\big[\textrm{Var}_\sigma[S(Y_i)]\big] \leq v_i
        \end{equation*}
        because the variance is invariant under constant translation. Also the second term is bounded below by $O(n^{-2})$; There are only three cases to consider for the variance calculation. For any $k$, $0 \leq i \leq k-1$, and $\sigma$, there is $\epsilon = \epsilon(\beta) > 0$ which satisfies
        \begin{equation*}
            \mathbb{P}_\sigma[S(Y_{i+1}) = S(Y_i) + 2/n] + \mathbb{P}_\sigma[S(Y_{i+1}) = S(Y_i) - 2/n] \in [\epsilon, 1-\epsilon].
        \end{equation*}
        Therefore there exists $c_1 > 0$ which satisfies
        \begin{equation} \label{eq: 2-8}
            v_{i+1} \leq v_i + \frac{c_1}{n^2}
        \end{equation}
        with $v_1 = \Omega(1/n^2)$. For the lower bound, analogously define $w_i := \min\textrm{Var}_\sigma[S(Y_i)]$ then
        \begin{equation} \label{eq: 2-9}
            w_{t+1} \geq w_t + \frac{c_2}{n^2}
        \end{equation}
        for some $c_2 = c_2(\beta) > 0$. Equation \eqref{eq: 2-8} and Equation \eqref{eq: 2-9} finishes the proof.
    \end{proof}
    
    We can derive a general variance bound by combining Equation \eqref{eq: 2-7} with Lemma \ref{lem: variance estimate no.1}. This can be done with the lemma below.
    
    \begin{lemma} \label{lem: variance estimate no.2}
        Suppose $(Z_t)$ be a (general) Markov chain which takes values from $\mathbb{R}$. Define $\mathbb{P}_z$ and $\mathbb{E}_z$ to be probability and expectation started from $z$. Suppose there exists some $0 < \rho < 1$ such that for any initial states $(z, z')$,
        \begin{equation}
            \Big|\mathbb{E}_z[Z_t] - \mathbb{E}_{z'}[Z_t]\Big| \leq \rho^t |z-z'|.
        \end{equation}
        Then $v_t := \sup_{z_0} \normalfont \textrm{Var}_{z_0}(Z_t)$ satisfies
        \begin{equation*}
            v_t \leq v_1 \min\big\{t, (1-\rho^2)^{-1}\big\}.
        \end{equation*}
    \end{lemma}

    \cite{LLP10} mentioned Lemma \ref{lem: variance estimate no.2} and its proof so we omit here. Combining Lemma \ref{lem: variance estimate no.1} and Lemma \ref{lem: variance estimate no.2}, we have Var$_\sigma[S_t] = O(1/n)$ at any time $t \geq 0$ if $\beta < 1$. The last tool for cutoff in high temperature regime is the information of the number of each spins on an arbitrary vertices subset.
    
    \begin{lemma} \label{lem: partial sum estimate}
            Suppose $\beta < 1$. \\
            (i) For all $\sigma \in \Omega$ and $v \in V$,
            \begin{equation*}
                |\mathbb{E}_\sigma[S_t]| \leq 2e^{-kt(1-\beta)/n} \qquad \textrm{and} \qquad |\mathbb{E}_\sigma[X_t(v)]| \leq 2e^{-kt(1-\beta)/n}.
            \end{equation*}
            (ii) For any subset $A \subset V$, define
            \begin{equation*}
                M_t(A) := \frac{1}{2}\sum_{v\in V}X_t(v).
            \end{equation*}
            Then we have $|\mathbb{E}_\sigma[M_t(A)]| \leq |A|e^{-kt(1-\beta)/n}$. Furthermore, if $t \geq [2(1-\beta)k]^{-1}n\log n$, then $\normalfont \textrm{Var}_\sigma[M_t(A)] \leq Cn$ for some constant $C>0$. \\
            (iii) For any subset $A$ of vertices and all $\sigma \in \Omega$,
            \begin{equation*}
                \mathbb{E}_\sigma\big[|M_t(A)|\big] \leq ne^{-kt(1-\beta)/n} + O(\sqrt{n}).
            \end{equation*}
        \end{lemma}
    \begin{proof}
        (i) Denote $\mathbbm{1}$ to be the configuration of all pluses. Consider the monotone coupling $(X_t, \widetilde{X}_t)$ where $X_0 = \mathbbm{1}$ and $\widetilde{X}_0$ following the stationary distribution $\mu$. Since the distribution on $\widetilde{X}_t$ remains the same and symmetric, $\mathbb{E}_{\mu}[\widetilde{S}_t]=0$ holds. From Proposition \ref{prop: magnetization contraction},
        \begin{equation*}
            \mathbb{E}_{\mathbbm{1}}[S_t] \leq \mathbb{E}_{\mathbbm{1}, \mu}\big[|S_t - \widetilde{S}_t|\big] + \mathbb{E}_\mu[\widetilde{S}_t] \leq 2\Big(1-\frac{k(1-\beta)}{n}\Big)^t \leq 2e^{-kt(1-\beta)/n}.
        \end{equation*}
        Similarly we consider $-\mathbbm{1}$, the configuration with all minuses, then $-2e^{-kt(1-\beta)/n} \leq \mathbb{E}_{-\mathbbm{1}}[S_t]$.
        The monotonicity ensures $\mathbb{E}_{-\mathbbm{1}}[S_t] \leq \mathbb{E}_{\sigma}[S_t] \leq \mathbb{E}_{\mathbbm{1}}[S_t]$. For $\mathbb{E}_\sigma[X_t(v)]$, the symmetry of $K_n$ gives
        \begin{equation*}
            \mathbb{E}_{\mathbbm{1}}[S_t] = \mathbb{E}_{\mathbbm{1}}[X_t(v)]
        \end{equation*}
        for any $v \in V$. Appealing to the monotonicity gives the second inequality.
        \medskip
        
        (ii) The expectation argument comes from (i). Consider the monotone coupling between $\mathbbm{1}$, $-\mathbbm{1}$, and $\sigma$ again. Denote $M_{\sigma, t}(A)$ as $M_t(A)$ starts from $\sigma$ then
        \begin{equation*}
            M_{-\mathbbm{1}, t}(A) \leq M_{\sigma, t}(A) \leq M_{\mathbbm{1}, t}(A).
        \end{equation*}
        As a consequence of Lemma \ref{lem: variance estimate no.2}, we have $\textrm{Var}_\sigma[S_t] = O(1/n)$. Hence $\mathbb{E}_\mathbbm{1}[M_t([n])^2] = n^2/4[\textrm{Var}_\mathbbm{1}[S_t] + (\mathbb{E}_\mathbbm{1}S_t)^2 = O(n)$ when $t \geq [2(1-\beta)]^{-1}n\log n$. Due to the symmetry of complete graphs,
        \begin{equation*}
            \mathbb{E}\Big[M_{\mathbbm{1}, t}([n])^2\Big] = n + {n \choose 2}\mathbb{E}[X_{\mathbbm{1}, t}(1)X_{\mathbbm{1}, t}(2)].
        \end{equation*}
        Therefore $|\mathbb{E}[X_{\mathbbm{1}, t}(1)X_{\mathbbm{1}, t}(2)]| = O(1/n)$. Consider the same expansion, then
        \begin{equation*}
            \mathbb{E}M_{\mathbbm{1}, t}(A)^2 = |A| + {|A| \choose 2}\mathbb{E}[X_{\mathbbm{1}, t}(1)X_{\mathbbm{1}, t}(2)] \leq n + {n \choose 2}O(1/n) = O(n).
        \end{equation*}
        Similar argument comes out for $\mathbb{E}M_{-\mathbbm{1}, t}(A)$. From the monotonicity,
        \begin{equation} \label{eq: 2-10}
            M_{\sigma, t}(A)^2 \leq M_{-\mathbbm{1}, t}(A)^2 + M_{\mathbbm{1}, t}(A)^2.
        \end{equation}
        Now taking expectation on both sides of Equation \eqref{eq: 2-10} we have $\mathbb{E}[M_{\sigma, t}(A)^2] = O(n)$. The first expectation argument implies $|\mathbb{E}M_{\sigma, t}(A)| = O(\sqrt{n})$ when $t \geq [2(1-\beta)k]^{-1}$, so we can get the variance estimate for large time $t$.
        \medskip
        
        (iii) Again, consider monotone coupling between $X_t$ and $\widetilde{X}_t$, where $X_0 = \sigma$ and $\widetilde{X}_0$ which follows the stationary distribution $\mu$. Then
        \begin{equation*}
            \mathbb{E}[|M_t(A)|] \leq \mathbb{E}[ |M_t(A) - \widetilde{M}_t(A)| ] + \mathbb{E}[|\widetilde{M}_t(A)|].
        \end{equation*}
        Applying Cauchy-Schwartz inequality considering $|\widetilde{M}_t(A) - M_t(A)| \leq \textrm{dist}(X_t, \widetilde{X}_t)$ yields
        \begin{equation*}
            \mathbb{E}[|M_t(A)|] \leq \mathbb{E}[ \textrm{dist}(X_t, \widetilde{X}_t)] + \sqrt{\mathbb{E}[\widetilde{M}_t(A)^2]}.
        \end{equation*}
        Applying Theorem \ref{prop: hamming distance contraction} subsequently gives
        \begin{equation*}
            \mathbb{E}[|M_t(A)|] \leq n\rho^t + \sqrt{\mathbb{E}[\widetilde{M}_t(A)^2]}.
        \end{equation*}
        Since the variables $\{\widetilde{X}_t(i)\}^n_{i=1}$ are positively correlated under $\mu$,
        \begin{equation*}
            \mathbb{E}[\widetilde{M}_t(A)^2] \leq \frac{n^2}{4}\mathbb{E}[\widetilde{S}^2_t] = \frac{n^2}{4}\textrm{Var}[\widetilde{S}_t] = O(n),
        \end{equation*}
        therefore we can get
        \begin{equation*}
            \mathbb{E}[|M_t(A)|] \leq ne^{-(1-\beta)kt/n} + O(\sqrt{n}).
        \end{equation*}
    \end{proof}
    
    Lemma \ref{lem: magnetization match to exact match} claims that for arbitrary two chains, if their magnetizations agree, then there is a coupling that they can agree in $O(n \log n / k)$ time with high probability. It becomes useful for $\beta \geq 1$ case.
    
    \begin{lemma} \label{lem: magnetization match to exact match}
        For any $\sigma, \widetilde{\sigma} \in \Omega$ such that $S(\sigma) = S(\widetilde{\sigma})$, there exists a coupling for $(X_t, \widetilde{X}_t)$ with starting points $(\sigma, \widetilde{\sigma})$,
        \begin{equation*}
            \lim_{n \to \infty} \mathbb{P}_{\sigma, \widetilde{\sigma}}\left[\min_t\{t : X_t = \widetilde{X}_t\} > \frac{\gamma n \log n}{k} \right] = 0,
        \end{equation*}
        for some constant $\gamma = \gamma(\beta) > 0$.
    \end{lemma}
    
    \begin{proof}
        For each time $t$ before updating $k$ vertices, rematch vertices of $X_t$ and $\widetilde{X}_t$ such that if $X_t(v) = \widetilde{X}_t(v)$, then match $v$ from $X_t$ with $v$ from $\widetilde{X}_t$. If $X_t(v) \neq \widetilde{X}_t(v)$, from $S_t = \widetilde{S}_t$ we have
        \begin{equation*}
            D_t := |\{v : (X_t(v),\widetilde{X}_t(v)) = (1, -1) \}| = |\{w : (X_t(w),\widetilde{X}_t(w)) = (-1, 1) \}|,
        \end{equation*}
        therefore we can rematch vertices one to one from $\{v : (X_t(v),\widetilde{X}_t(v)) = (1, -1) \}$ to $\{w : (X_t(w),\widetilde{X}_t(w)) = (-1, 1) \}$. We consider monotone coupling on these chains, with the rematched vertices. Their magnetizations remain agreed under the coupling, and the number of vertices whose spins are not matched, $2D_t$, decreases. More precisely, $D_t$ is supermartingale, and there exists $c = c(\beta)>0$ such that
        \begin{equation*}
            \mathbb{E}_{\sigma, \widetilde{\sigma}}[D_{t+1} | D_t] \leq \Big(1 - \frac{ck}{n}\Big)D_t.
        \end{equation*}
        Therefore, combining with the Markov inequality, we get
        \begin{equation*}
            \mathbb{P}_{\sigma, \widetilde{\sigma}}\left[\min_t\{t : X_t = \widetilde{X}_t\} > \frac{\gamma n \log n }{ k} \right] \leq \mathbb{P}_{\sigma, \widetilde{\sigma}}[D_{\gamma n \log n / k} \geq 1] \leq n\exp(-c\gamma \log n).
        \end{equation*}
        Choosing proper $\gamma = \gamma(\beta)$ makes the rightmost term go to $0$ as $n \to \infty$.
    \end{proof}

    It is well known that the magnetization under the original Glauber dynamics tends to go to zero, so the absolute value of it becomes supermartingale. Similarly we need a bound to measure how much the magnetization changes by randomized systematic scan dynamics.
    
    \begin{lemma} \label{lem: apriori estimate} (Magnetization estimate)
        For any given chain $X_t$ starting from $\sigma$,
        \begin{equation} \label{eq: 2-11}
            \Big| \mathbb{E}_\sigma[S_1] - \big(1-\frac{k}{n}\big)S(\sigma) - \frac{k}{n}\tanh \beta S(\sigma) \Big| \leq \frac{2k}{n}\tanh\beta\frac{2k}{n} = O\Big(\frac{k^2}{n^2}\Big).
        \end{equation}
    \end{lemma}
        
    \begin{proof}
        Consider the intermediate states $Y_0 = \sigma, ..., Y_k = X_1$. For any vertex selected to be updated, the probability that the new spin becomes $1$ is
        \begin{equation*}
            p_+\Big(S(Y_i) \pm \frac{1}{n}\Big) = \frac{1 + \tanh \beta (S(Y_i) \pm 1/n)}{2},
        \end{equation*}
        thus the expectation of the updated spin of the vertex is
        \begin{equation*}
            p_+\Big(S(Y_i) \pm \frac{1}{n}\Big) - p_-\Big(S(Y_i) \pm \frac{1}{n}\Big) = \tanh \beta \Big(S(Y_i) \pm \frac{1}{n}\Big)
        \end{equation*}
        and we know $|S(Y_i) - S_0| \leq \frac{2i}{n}$.
        
        Since $k$ vertices are chosen uniformly randomly among $n$ vertices at $t=0$, the expected sum of spins of vertices which are not chosen to be updated is $(n-k)S_0$. Therefore
        \begin{equation*}
            \begin{split}
                &(n-k)S_0 + k \tanh \beta \Big(S_0 - \frac{1}{n} - \frac{2(k-1)}{n}\Big) \\
                &\qquad \qquad \qquad \leq n\mathbb{E}_\sigma[S_1] \leq (n-k)S_0 + k \tanh \beta \Big(S_0 + \frac{1}{n} + \frac{2(k-1)}{n}\Big),
            \end{split}
        \end{equation*}
        which becomes Equation \eqref{eq: 2-11} by appealing $|\tanh x - \tanh y| \leq 2\tanh \left|(x-y)/2\right|$.
    \end{proof}

    Sometimes we need to investigate the magnetization not only for $(X_t)$ but for all its intermediates $(Y_i)$. Since $|S(Y_{i+1}) - S(Y_i)| \leq 2/n$, we can get an estimate for the time $\tau_\textrm{mag}:= \min\{i \geq 0 : S(Y_i) = S(\widetilde{Y}_i)\}$. However, due to unavailable vertices, we are not sure whether we can manipulate the coupling so that those two chains' magnetization remain equal. The following lemma partially gives an answer for such a case.

    \begin{lemma} \label{lem: crossing is almost coupling}
        Assume $k=o(\sqrt{n})$. Suppose we have two chains $X_t$ and $\widetilde{X}_t$(not necessarily independent) such that $S_0 > \widetilde{S}_0$. Define their intermediate states $\{Y_i\}$ and $\{\widetilde{Y}_i\}$ such that $(Y_{kt}, \widetilde{Y}_{kt}) = (X_t, \widetilde{X}_t)$ for all $t\geq 0$. Consider $\tau_\textrm{almost}:= \min \{i \geq 0: |S(Y_i)-S(\widetilde{Y_i})|\leq 2/n\}$ and $\tau_\textrm{exact}:= \min \{i \geq 0: S(Y_i)=S(\widetilde{Y_i})\}$ under any given coupling. Then,
        \begin{enumerate}[leftmargin=*]
        \item There exists a coupling for intermediates from $\tau_\textrm{exact}$ to $T := \lceil\tau_\textrm{exact}/k\rceil$ which ensures $S_T = \widetilde{S}_T$ with high probability.
        \item There exists a coupling for intermediates from $\tau_\textrm{almost}$ to $T' := \lceil\tau_\textrm{almost}/k\rceil$ which ensures that $|S_{T'} - \widetilde{S}_{T'}| \leq 4k/n$ with high probability, for large enough $n$.
        \end{enumerate}
    \end{lemma}
    
    \begin{remark*}
        Both $\tau_\textrm{almost}$ and $\tau_\textrm{exact}$ are defined on intermediates. Lemma \ref{lem: crossing is almost coupling} implies that if the two magnetizations on intermediates are matched, magnetizations on the original chains can be matched soon with high probability. For the case (ii), Lemma \ref{prop: magnetization contraction} and Markov inequality ensures the magnetization matching after $O(n/k)$ times.
    \end{remark*}
                        
    \begin{proof}
        (i) Let $l := \tau_\textrm{exact} \textrm{ (mod k)}$. At time $i = \tau_\textrm{exact}$, rematch vertices according to their current spins without considering its update history. For the remaining $(k-l)$ updates, apply monotone coupling for $(Y_i, \widetilde{Y}_i)$: randomly choose one of the available vertices from $Y_i$ and check the corresponding vertices of $\widetilde{Y}_i$. If that vertex from $\widetilde{Y}_i$ is an unavailable vertex, stop monotone coupling and run two Markov chains independently. If not, update two corresponding vertices under rematched monotone coupling.
            
        This coupling preserves the magnetization matching unless it choose unavailable vertices during $(k-l)$ updates. In the worst case, the probability that the monotone coupling does not stop at $i = j$ is $(n-2j)/(n-j) = 1 - j/(n-j)$. After it survives, the probability that the coupling does not stop at $i = (j+1)$ is $(n-2j-1)/(n-j-1) = 1 - j/(n-j-1)$. To summarize, the probability that the coupling does not stop until $i = k$ is greater than or equal to
        \begin{equation*}
            \Big(1 - \frac{l}{n-l}\Big)\Big(1 - \frac{l}{n-l-1}\Big)...\Big(1 - \frac{l}{n-k}\Big).
        \end{equation*}
        This is greater than
        \begin{equation}
            \Big(1 - \frac{k}{n-k}\Big)^{k} \sim \exp\left(-\frac{k^2}{n-k}\right) \to 1.
        \end{equation}
        Since $k = o(\sqrt{n})$, there is a coupling that ensures $\mathbb{P}[S_T = \widetilde{S}_T] \to 1$ as $n \to \infty$. \\
        (ii) In case of $\tau_\textrm{almost} = \tau_\textrm{exact}$, there is nothing to prove. Without loss of generality, suppose $S(Y_{\tau_\textrm{almost}}) - S(\widetilde{Y}_{\tau_\textrm{almost}}) = 2/n$. Rematch vertices so that $Y_{\tau_\textrm{almost}}$ and $\widetilde{Y}_{\tau_\textrm{almost}}$ are monotone. Apply monotone coupling from $\tau_\textrm{almost}$ to $kT'$ in the same way as (i): stop the coupling and update independently when an unavailable vertex is chosen to be updated. Monotone coupling succeeds to time $kT'$ with probability $1 - O(k^2/n)$. Only one vertex is mismatched at time $\tau_\textrm{almost}$ so an analogous argument from Proposition \ref{prop: hamming distance contraction} ensures that
        \begin{equation*}
            |S_{T'} - \widetilde{S}_{T'}| \leq \frac{2}{n}\left(1 + \frac{\beta}{n}\right)^{k-l} \leq \frac{4}{n}.
        \end{equation*}
        for large enough $n$.
    \end{proof}

\section{Main results: high temperature regime}
    In this section, we prove Theorem \ref{thm: main1}. Assume $\beta < 1$ throughout this section.
    
\subsection{Mixing time upper bound for $\beta < 1$}
    Here we set the exact upper bound statement for the mixing time first.
    \begin{theorem} \label{thm: main1 upper bound}
        Suppose \textbf{$k=o(\sqrt[3]{n})$}. Then
        \begin{equation}
            \lim_{\gamma \to \infty} \limsup_{n\to\infty} d_n \Big(\frac{n\log n}{2k(1-\beta)} +  \frac{\gamma n}{k} \Big) = 0.
        \end{equation}
    \end{theorem}
    
    The proof of Theorem \ref{thm: main1 upper bound} consists of two parts: for any given two configuration chains we couple their magnetizations in the first part(magnetization matching phase), then set up another coupling so that their spins exactly match with high probability(two-coordinate chain phase). Each parts consists of two small steps, since two chains often cross each other during $k$ updates. For any two chains $X_t$ and $\widetilde{X}_t$,
    \begin{enumerate}[leftmargin=*]
        \item Update $X_t$ and $\widetilde{X}_t$ under the grand coupling. In $\frac{n \log n}{2k(1-\beta)}$ time, $|S(X_t) - S(\tilde{X}_t)| \leq \frac{4k}{n}$ holds with high probability.
        \item Rematch vertices of $X_t$ and $\widetilde{X}_t$ and run two chains independently. After additional $\gamma_1 n/k$ time, $S(X_t) = S(\tilde{X}_t)$ holds with high probability.
        \item Setup a coupling so that $S(X_t) = S(\widetilde{X}_t)$ remains true. Define two-coordinate chain from $X_t$ and $\widetilde{X}_t$ (call them $U^X_t$ and $U^{\tilde{X}}_t$). After additional $\gamma_2 n/k$ time, $|U^X_t - U^{\tilde{X}}_t| \leq k$ holds with high probability.
        \item Setup another coupling for $U^X_t$ and $U^{\widetilde{X}}_t$. After additional $n/k$ time, $U^X_t = U^{\widetilde{X}}_t$ holds with high probability, and this becomes equivalent to $X_t = \widetilde{X}_t$.
    \end{enumerate}
    In particular, (1), (2), and (3) requires the restriction $k = o(\sqrt{n})$. $k = o(\sqrt[3]{n})$ is necessary only for (4).
                
    \begin{theorem} (magnetization matching phase) \label{thm: magnetization matching phase} \\
        Suppose $k = o(\sqrt{n})$. For any two configurations $\sigma$ and $\widetilde{\sigma}$, there exists a coupling $(X_t, \widetilde{X}_t)$ starting with $X_0 = \sigma$ and $\widetilde{X}_0 = \widetilde{\sigma}$, which satisfies $S_{t^*} = \widetilde{S}_{t^*}$ with probability $1 - O(\gamma^{-1/2})$, where $t^* = [2k(1-\beta)]^{-1}n \log n + \gamma n/k$.
        \end{theorem}

    \begin{proof}
        For convenience, define $t(\gamma) := [2k(1-\beta)]^{-1}n\log n + \gamma n/k$. For any given configuration $\sigma$ and $\widetilde{\sigma}$ consider the monotone coupling $(X_t, \widetilde{X}_t)$ with $X_0 = \sigma$ and $\widetilde{X}_0 = \widetilde{\sigma}$. From Proposition \ref{prop: magnetization contraction},
        \begin{equation} \label{eq: 3-1}
            \mathbb{E}_{\sigma, \widetilde{\sigma}}\left[|S_{t(0)} - \widetilde{S}_{t(0)}|\right] \leq c_1 n^{-1/2}
        \end{equation}
        holds for some $c_1 > 0$. Without loss of generality assume $S_{t(0)} \geq \widetilde{S}_{t(0)}$ and define a stopping time
        \begin{equation*}
            \tau_\textrm{pre} := \min\{t \geq t(0) : |S_t - \widetilde{S}_t| \leq 4k/n \}
        \end{equation*}
        If $\tau_\textrm{pre} \leq t(\gamma)$, from time $\tau_\textrm{pre}$ rematch vertices according to their spins and update together under the monotone coupling. Otherwise run two chains $(X_t)$ and $(\widetilde{X}_t)$ under the monotone coupling until $t(\gamma)$ and run independently for $t(\gamma) < t \leq \tau_\textrm{pre}$. 
        Since $S_t \geq \widetilde{S}_t$ for $t \leq \tau_\textrm{pre}$, the process $(S_t - \widetilde{S}_t)_{t(\gamma) \leq t < \tau_\textrm{pre}}$ has a non-positive drift and is non-negative. $t < \tau_\textrm{pre}$ ensures $S_{t+1} - \widetilde{S}_{t+1} > 0$ so Lemma \ref{lem: supermartingale lemma} and Lemma \ref{lem: variance estimate no.1} can be applied, provided that $k = o(\sqrt{n})$,
        \begin{equation*}
            \mathbb{P}_{\sigma, \widetilde{\sigma}}[\tau_\textrm{pre} > t(\gamma)| X_{t(0)}, \widetilde{X}_{t(0)}] \leq \frac{c|S_{t(0)} - \widetilde{S}_{t(0)}|}{\frac{\sqrt{k}}{n}\sqrt{\frac{\gamma n}{k}}} = \frac{c\sqrt{n}|S_{t(0)} - \widetilde{S}_{t(0)}|}{\sqrt{\gamma}}.
        \end{equation*}
        Taking expectation over $X_{t(0)}$ and $\widetilde{X}_{t(0)}$ with Equation \eqref{eq: 3-1} gives
        \begin{equation} \label{eq: 3-2}
            \mathbb{P}_{\sigma, \widetilde{\sigma}}[\tau_\textrm{pre} > t(\gamma_1)] \leq O(\gamma^{-1/2}_1).
        \end{equation}
        After we reach to $\tau_\textrm{pre}$, the number of plus signs of $X_{\tau_\textrm{pre}}$ and $\widetilde{X}_{\tau_\textrm{pre}}$ are different by at most $2k$, by the definition of $\tau_\textrm{pre}$. Define another stopping time
        \begin{equation*}
            \tau_\textrm{almost} := \min\{i \geq k\tau_\textrm{pre}: |S(Y_i) - S(\widetilde{Y}_i)| \leq 2/n\}.
        \end{equation*}
        From Lemma \ref{lem: crossing is almost coupling}, there is a coupling from $\tau_\textrm{almost}$ to $k\lceil \tau_\textrm{almost}/k\rceil$ with high probability that ensure the magnetization difference is at most $4/n$. Without loss of generality suppose $S(Y_{k\tau_\textrm{pre}}) > S(\widetilde{Y}_{k\tau_\textrm{pre}})$. Consider a sequence $\{S_t - \widetilde{S}_t + 4k/n\}_{t \geq \tau_\textrm{pre}}$ by running $X_t$ and $\widetilde{X}_t$ independently, then we can apply Lemma \ref{lem: supermartingale lemma} for $\tau := \min\{t \geq k\tau_\textrm{pre} : S_t - \widetilde{S}_t + 4k/n \geq 4k/n\}$:
        \begin{equation*}
            \mathbb{P}_{\sigma, \widetilde{\sigma}}[\tau > \tau_\textrm{pre} + \gamma_2 n/k] \leq \frac{4\frac{8k}{n}}{\frac{\sqrt{k}}{n}\sqrt{\gamma_2 n}} = \frac{32 \sqrt{k}}{\sqrt{\gamma_2 n}} = O(\gamma_2^{-1/2}).
        \end{equation*}
        However $k\tau \geq \tau_\textrm{almost}$, hence 
        \begin{equation} \label{eq: 3-3}
            \mathbb{P}_{\sigma, \widetilde{\sigma}}[\tau_\textrm{almost} > k\tau_\textrm{pre} + \gamma_2 n] \leq O(\gamma_2^{-1/2}).
        \end{equation}
        From $\tau_\textrm{almost}$ we couple intermediates as we discussed in Lemma \ref{lem: crossing is almost coupling}. Then we can have at time $t = T := k\lceil\tau_\textrm{almost}/k\rceil$ the magnetization difference is less then or equal to $4/n$ with high probability. If it is $0$ then the magnetization matching phase is done. If not, we can rematch vertices from time $t = T$ so that $X_T(v)-\widetilde{X}_T(v)$ are either all non-negative or all non-positive. Then Proposition \ref{prop: magnetization contraction} and Markov inequality ensures
        \begin{equation} \label{eq: 3-4}
            \begin{split}
                \mathbb{P}_{\sigma, \widetilde{\sigma}}[S_{t + T} = \widetilde{S}_{t + T} | \mathcal{F}_T] &\leq n\mathbb{E}_{\sigma, \widetilde{\sigma}}\left[|S_{t + T} - \widetilde{S}_{t + T}| \big| |S_T - \widetilde{S}_T|\right] \\
                &\leq n\left( 1-\frac{k(1-\beta)}{n}\right)^t |S_T - \widetilde{S}_T| \leq 4e^{-\frac{kt(1-\beta)}{n}}
            \end{split}
        \end{equation}
        Hence $t = \gamma_3 n/k$ is enough to assure $S_{t + T} = \widetilde{S}_{t + T}$ with high probability. To sum up, define $\tau_\textrm{mag} := \min\{t\geq 0: S_t = \widetilde{S}_t\}$ then combining Equation \eqref{eq: 3-2}, \eqref{eq: 3-3}, and \eqref{eq: 3-4} gives
        \begin{equation}
            \begin{split}
            &\mathbb{P}_{\sigma, \widetilde{\sigma}}[\tau_\textrm{mag} < t(3\gamma)] \\
            &\qquad \geq \mathbb{P}_{\sigma, \widetilde{\sigma}}[\tau_\textrm{pre} < t(\gamma)] \mathbb{P}_{\sigma, \widetilde{\sigma}}[\tau_\textrm{almost} < k\tau_\textrm{pre} + \gamma n] \mathbb{P}_{\sigma, \widetilde{\sigma}}[\tau_\textrm{mag} < \lceil\tau_\textrm{almost}/k\rceil + \gamma n/k] \\
            &\qquad = 1 - O(\gamma^{-1/2}) \to 1.
            \end{split}
        \end{equation}
        as $\gamma \to \infty$.
    \end{proof}

    After we reach $\tau_\textrm{mag}$, we transform two chains $(X_t, \widetilde{X}_t)$ into another Markov chain $(U_t, \widetilde{U}_t)$ and try to setup a coupling on $U_t$. Following lemmas are necessary to move on to the two-coordinate chains $(U_t, \widetilde{U}_t)$. We omit the proof of Lemma \ref{lem: good starting points}; See \cite{LLP10} page 239 for the proof.

    \begin{lemma} \label{lem: good starting points}
        Suppose $k = o(n)$. For any subset $\Omega_0 \subset \Omega = \{-1, 1\}^n$ with stationary distribution $\mu$,
    \begin{equation*}
            \max_{\sigma \in \Omega} \| \mathbb{P}_\sigma(X_{t_0 + t}\in \cdot ) - \mu \|_{TV} \leq \max_{\sigma_0 \in \Omega_0} \| \mathbb{P}_{\sigma_0}(X_{t}\in \cdot ) - \mu \|_{TV} + \max_{\sigma \in \Omega} \mathbb{P}_\sigma(X_{t_0} \notin \Omega_0).
    \end{equation*}
    \end{lemma}
    
    Consider the set $\Omega_0 = \{\sigma\in\Omega : |S(\sigma)| \leq 1/2\}$. From Lemma \ref{lem: partial sum estimate}, there is a constant $\theta_0>0$ such that $|\mathbb{E}_\sigma[S_{\theta_0 n/k}]|\leq 1/4$. Since $k = o(n)$, 
        \begin{equation*}
            \begin{split}
                \mathbb{P}_\sigma(X_{\theta_0 n/k} \notin \Omega_0) &= \mathbb{P}_\sigma(|S_{\theta_0 n/k}| > 1/2) \\
                &\leq \mathbb{P}_\sigma(|S_{\theta_0 n/k} - \mathbb{E}_\sigma[S_{\theta_0 n/k}]| > 1/4) \leq 16\textrm{Var}_\sigma(S_{\theta_0 n/k}) = O(n^{-1}).
            \end{split}
        \end{equation*}
        
    The last equation comes from the consequence of Lemma \ref{lem: variance estimate no.2}. Both the number of positive spins and that of negative spins from any $\sigma_0 \in \Omega_0$ are in between $n/4$ and $3n/4$. More formally, define
        \begin{equation*}
            u_0 := |\{v \in V : \sigma_0(v) = 1\}|, \qquad v_0 := |\{v \in V : \sigma_0(v) = -1\}|
        \end{equation*}
        as numbers of each spins, and
        \begin{equation*}
            \Lambda_0 := \{(u, v) \in \mathbb{Z}^2 : n/4 \leq u, v \leq 3n/4, u+v = n\}.
        \end{equation*}
        Then $\sigma_0 \in \Omega_0$ if and only if $(u_0, v_0) \in \Lambda_0$. With this definition we move on two-coordinate chain $(u, v)$ from the original chain $(X_t)$.
        
        \begin{definition} (two-coordinate chain)
            Fix a configuration $\sigma_0 \in \Omega_0$. For $\sigma \in \Omega$, define
            \begin{equation*}
                \begin{split}
                U_{\sigma_0}(\sigma)&:= |\{v \in V: \sigma(v) = \sigma_0(v) = 1\}| \\
                V_{\sigma_0}(\sigma)&:= |\{v \in V: \sigma(v) = \sigma_0(v) = -1\}|.
                \end{split}
            \end{equation*}
        \end{definition}
        From now on, we shall write simply $U(\sigma)$ for $U_{\sigma_0}(\sigma)$ and $V(\sigma)$ for $V_{\sigma_0}(\sigma)$. \\
        For any randomized systematic scan dynamics chain $\{X_t\}$, we can define a process $(U_t, V_t)_{t\geq 0}$ by
        \begin{equation*}
            U_t = U(X_t) \qquad \textrm{and} \qquad V_t = V(X_t).
        \end{equation*}
        This is a Markov chain on $\{0, ..., u_0\} \times \{0, ..., v_0\}$. Denote the stationary measure for this chain as $\pi_2$, and note that this chain also determines the magnetization of the original chain $\{X_t\}$:
        \begin{equation*}
            S_t = \frac{2(U_t - V_t)}{n} - \frac{u_0 - v_0}{n}.
        \end{equation*}
        The first thing to check is how the original Markov chain and two-coordinate chain are related to each other. The total variation distances of the original chain and that of the two-coordinate chain turn out to be equal. This is also from \cite{LLP10}, page 241. We omit the proof.
        
        \begin{lemma} \label{lem: distance isometry}
            Suppose $(X_t)$ is the Glauber dynamics that starts from $\sigma_0$ and $(U_t, V_t)$ is the corresponding two-coordinate chain that starts from $(u_0, v_0)$. Then
            \begin{equation*}
                \|\mathbb{P}_{\sigma_0}(X_t \in \cdot) - \mu\|_{TV} = \|\mathbb{P}_{(u_0, v_0)}((U_t, V_t) \in \cdot) - \pi\|_{TV},
            \end{equation*}
            where $\mu$ and $\pi$ are the stationary distributions of two chains, respectively.
        \end{lemma}
        
        Thanks to Lemma \ref{lem: distance isometry}, it suffices to bound from above the total variation distance of the two-coordinate chain. Let's determine $\sigma_0$ later and consider the two coordinate chain first. The difference of $U$ becomes non-negative supermartingale for a short period of time.
        
        \begin{lemma} \label{lem: two-coordinate chain closes}
            Suppose two configuration $\sigma$ and $\widetilde{\sigma}$ satisfy $S(\sigma) = S(\widetilde{\sigma})$ and $U(\widetilde{\sigma}) - U(\sigma) > 0$. Define
            \begin{equation*}
                \Xi_1 := \Big\{\sigma^* \in \Omega : \min\{U(\sigma^*), u_0 - U(\sigma^*), V(\sigma^*), v_0 - V(\sigma^*)\}\geq \frac{n}{16} \Big\}
            \end{equation*}
            and
            \begin{equation*} \label{eq: 8-4}
                R(\sigma_1, \sigma_2):=U(\sigma_1)-U(\sigma_2), \qquad \qquad \tau_k = \min_{t\geq 0} \{t : R(\widetilde{X}_t, X_t) \leq k\}.
            \end{equation*}
            There exists a Markovian coupling $(X_t, \widetilde{X}_t)_{0 \leq t \leq \tau_k}$ of the randomized systematic scan dynamics with initial state $X_0 = \sigma$ and $\widetilde{X}_t = \tilde{\sigma}$, such that
            \begin{enumerate}[leftmargin=*]
                \item $S_t = \widetilde{S}_t$ at any time $t$.
                \item For all $t \leq \tau_k$, intermediate states $\{Y_i\}_{0\leq i\leq k\tau_k}$ and $\{\widetilde{Y}_i\}_{0\leq i \leq k\tau_k}$ such that $Y_{ik} = X_i$ and $\widetilde{Y}_{ik} = \widetilde{X}_i$ for all $i\geq 0$ satisfies
                \begin{equation*}
                    \mathbb{E}_{\sigma, \tilde{\sigma}} \left[R(\widetilde{Y}_{i+1}, Y_{i+1}) - R(\widetilde{Y}_i, Y_i) \Big| Y_i, \widetilde{Y}_i\right] \leq 0,
                \end{equation*}
                for any $i$.
                \item There exists a constant $c>0$ independent from $k$ and $n$ so that on the event $\{X_t \in \Xi, \widetilde{X}_t \in \Xi\}$,
                \begin{equation*}
                    \mathbb{P}_{\sigma, \tilde{\sigma}}\left[R(\widetilde{Y}_{i+1}, Y_{i+1}) - R(\widetilde{Y}_i, Y_i) \neq 0 \Big| Y_i, \widetilde{Y}_i\right] \geq c.
                \end{equation*}
            \end{enumerate}
        \end{lemma}
        
        \begin{remark*}
            The lemma ensures the chain $R(\widetilde{Y}_i, Y_i)$ to be supermartingale until it becomes less than $k$. Note that the statement works on all intermediate states, not only for $R(\widetilde{X}_t, X_t)$.
        \end{remark*}
        
        \begin{proof}
            We couple $(X_t, \widetilde{X}_t)$ by coupling $(Y_i, \widetilde{Y}_i)_{0\leq i\leq k\tau_k}$ for every $i$. From $(Y_i, \widetilde{Y}_i)$ to $(Y_{i+1}, \widetilde{Y}_{i+1})$,  choose a vertex $I_i$ randomly from $V \setminus \mathcal{N}_i$ and assign the next spin $\mathcal{S}$ according to the probability from equation \eqref{eq: 2-1};
            \begin{equation*}
                \mathbb{P}(\mathcal{S} = 1) = p_{+}(S(Y_i) - Y_i(I_i)/n).
            \end{equation*}
            In other words, $Y_{i+1}$ becomes
            \begin{equation*}
                Y_{i+1}(v) = 
                \begin{cases}
                    Y_i(v) & v \neq I_i, \\
                    \mathcal{S} & v = I_i.
                \end{cases}
            \end{equation*}
            For $\widetilde{Y}_{i+1}$, we select $\widetilde{I}_i$ randomly from $\{v : \widetilde{Y}_i(v) = Y_i(v)\} \setminus \widetilde{\mathcal{N}}_i$  and set
            \begin{equation*}
                \widetilde{Y}_{i+1}(v) = 
                \begin{cases}
                    \widetilde{Y}_i(v) & v \neq \widetilde{I}_i, \\
                    \mathcal{S} & v = \widetilde{I}_i.
                \end{cases}
            \end{equation*}

            \begin{figure}[h]
                \includegraphics[scale = 0.45]{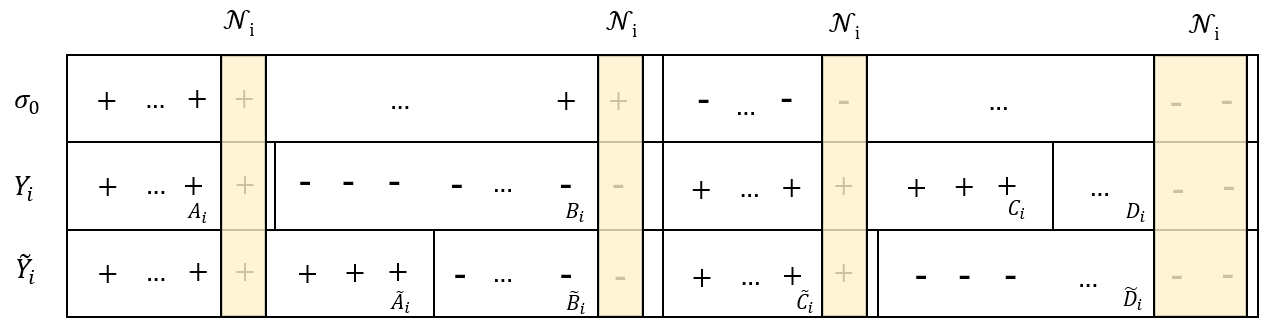}
                \centering
                \caption{All vertices except $\mathcal{N}_i$(which are not allowed to be updated at the moment) are divided into four categories. Some vertices from $B_i \cap \widetilde{A}_i$ and $C_i \cap \widetilde{D}_i$ can be also included in $\mathcal{N}_i$.}
                \label{fig: coupling}
            \end{figure}
            
            The first condition of Lemma \ref{lem: two-coordinate chain closes} is done since this coupling ensures $S(Y_i) = S(\widetilde{Y}_i)$ for any $i$. Next, consider
            \begin{equation*}
                \begin{split}
                    A_i &= |\{v : \sigma_0(v) = 1, Y_i(v) = 1\} \setminus \mathcal{N}_i| \\
                    B_i &= |\{v : \sigma_0(v) = 1, Y_i(v) = -1\} \setminus \mathcal{N}_i| \\
                    C_i &= |\{v : \sigma_0(v) = -1, Y_i(v) = 1\} \setminus \mathcal{N}_i| \\
                    D_i &= |\{v : \sigma_0(v) = -1, Y_i(v) = -1\} \setminus \mathcal{N}_i|,
                \end{split}
            \end{equation*}
            and $\widetilde{A}_i$, $\widetilde{B}_i$, $\widetilde{C}_i$, $\widetilde{D}_i$ analogous way. From $S(Y_i) = S(\widetilde{Y}_i)$ we have $A_i + C_i = \widetilde{A}_i + \widetilde{C}_i$ and $B_i + D_i = \widetilde{B}_i + \widetilde{D}_i$.

            Denote $R_i := R(\widetilde{Y}_i, Y_i)$, and $i' = i \textrm{(mod k)}$ for convenience. Then
            \begin{equation*}
                \begin{split}
                \mathbb{P}_{X_0, \widetilde{X}_0}[R_{i+1} - R_i = -1 | Y_i, \widetilde{Y}_i]
                &= \frac{C_i}{n - i'}\frac{\widetilde{A}_i}{\widetilde{A}_i + \widetilde{C}_i}p_-\Big(S(Y_i) - \frac{1}{n}\Big) + \frac{B_i}{n-i'}\frac{\widetilde{D}_i}{\widetilde{B}_i + \widetilde{D}_i}p_+\Big(S(Y_i) + \frac{1}{n}\Big) \\
                \mathbb{P}_{X_0, \widetilde{X}_0}[R_{i+1} - R_i = 1 | Y_i, \widetilde{Y}_i]
                &= \frac{A_i}{n - i'}\frac{\widetilde{C}_i}{\widetilde{A}_i + \widetilde{C}_i}p_-\Big(S(Y_i) - \frac{1}{n}\Big) + \frac{D_i}{n-i'}\frac{\widetilde{B}_i}{\widetilde{B}_i + \widetilde{D}_i}p_+\Big(S(Y_i) + \frac{1}{n}\Big).
                \end{split}
            \end{equation*}
            Since $R_{i+1} - R_i \in \{-1, 0, 1\}$,
            \begin{equation*}
                \mathbb{E}_{X_0, \widetilde{X}_0}[R_{i+1} - R_i | Y_i, \widetilde{Y}_i] = 
                \frac{\widetilde{C}_i - C_i}{n - i}p_-\Big(S(Y_i) - \frac{1}{n}\Big) + \frac{\widetilde{B}_i - B_i}{n - i}p_+\Big(S(Y_i) + \frac{1}{n}\Big).
            \end{equation*}
            For any $t \leq \tau_k - 1$,
            \begin{equation*}
                k \leq R_{kt} = -(\widetilde{C}_{kt} - C_{kt}) = -(\widetilde{B}_{kt} - B_{kt}),
            \end{equation*}
            hence $\widetilde{C}_i - C_i$ and $\widetilde{B}_i - B_i$ are both less then or equal to $0$ if $i \leq k\tau_k$ because both of them can only change by $1$ as $i$ changes. The last condition of Lemma \ref{lem: two-coordinate chain closes} follows from that $p_\pm([-1, 1]) \subset [\epsilon, 1-\epsilon]$ for some $\epsilon>0$.
        \end{proof}
        
        Lemma \ref{lem: two-coordinate chain closes} satisfies all the conditions of Lemma \ref{lem: supermartingale lemma}, thus we can calculate how much time we need for $\tau_k$. Without loss of generality let $U_{\tau_k} < \widetilde{U}_{\tau_k}$. At each time $t \geq \tau_k$, rematch the vertices of $A_{tk}$ to $\widetilde{A}_{tk}$ as many as possible, and do the same procedure for $D_{tk}$ to $\widetilde{D}_{tk}$. Match $\widetilde{B}_{tk}$ to $B_{tk}$, and $\widetilde{C}_{tk}$ to $C_{tk}$ as well. Match the remaining vertices of $C_{tk}$ to those of $\widetilde{A}_{tk}$, and those of $B_{tk}$ to $\widetilde{D}_{tk}$. All vertices from $X_t$ and $\widetilde{X}_t$ are matched one to one; Now update $k$ vertices arbitrarily from $X_t$ and corresponding vertices from $\widetilde{X}_t$ under the Monotone coupling. Repeat the rematching process according to their current spins and updates, unless $U_t = \widetilde{U}_t$. This coupling has two remarks. First of all, magnetizations remain coupled during the process, and secondly $R_t$ is non-negative and non-increasing sequence.
        
        \begin{lemma} \label{lem: tau_k to tau_match}
            Suppose $k = o(n^{1/3})$. Define $\tau_{\normalfont\textrm{match}} = \min \{t \geq \tau_k : U_t = \widetilde{U}_t\}$. Then we have
            \begin{equation}
                \mathbb{P}_{X_{\tau_k}, \widetilde{X}_{\tau_k}}\Big\{\tau_{\normalfont\textrm{match}} > \tau_k + \frac{\gamma n}{k}\Big\} = O(k^{3/2}n^{-1/2}).
            \end{equation}
        \end{lemma}

        \begin{proof}
            Suppose we apply the original (single site) Glauber dynamics to $X_{\tau_k}$ and $\widetilde{X}_{\tau_k}$ and let them $X'_t$ and $\widetilde{X}'_t$. Then $\{R(X'_i, \widetilde{X}'_i)\}_{i \geq k\tau_k}$ becomes supermartingale. Lemma \ref{lem: supermartingale lemma} ensures that the stopping time $\tau'_R := \{i \geq k\tau_k: R(X'_i, \widetilde{X}'_i) = 0\}$ satisfies $\mathbb{P}_{\sigma, \widetilde{\sigma}}[\tau'_R > k\tau_k + A | R_{k\tau_k}] \lesssim \frac{k}{\sqrt{A}}$ for any large enough $A$. Now compare this result with systematic scan dynamics case. Define $\tau_R := \{i \geq k\tau_k: R_i = R(X_i, \widetilde{X}_i) = 0\}$. In the similar fashion as Lemma \ref{lem: crossing is almost coupling}, grouping by $k$ numbers of single site updates from $i = \tau_k$ to $i = \tau_k + A$ generates $A / k$ times of systematic scan update and the probability that none of the group update same vertex more than once equals greater than $(e^{-k^2/n})^{A / k} = e^{-Ak/n}$. Hence
            \begin{equation*}
                \begin{split}
                    \mathbb{P}_{\sigma, \widetilde{\sigma}}[\tau_R \leq k\tau_k + A | R_{k\tau_k}] &\geq \mathbb{P}[\tau'_R = \tau_R]\mathbb{P}_{\sigma, \widetilde{\sigma}}[\tau'_R \leq k\tau_k + A | R_{k\tau_k}] \\
                    &\geq e^{-Ak/n}\left(1 - O\left(\frac{k}{\sqrt{A}}\right)\right).
                \end{split}
            \end{equation*}
            Pick $A = \sqrt{nk}$, then the rightmost term becomes $1 - O(k^{3/2}n^{-1/2})$, which goes to $1$. For the original dynamics, after $i = \tau'_R$, $R(X'_i, \widetilde{X}'_i)$ remains zero. Hence modifying the argument for the smallest multiple of $k$ greater than $\tau'_R$ would give the same result on $\tau_\textrm{match}$. The proof is finished by $\mathbb{P}[\tau_{\normalfont\textrm{match}} \geq \tau_k + \sqrt{n/k}] \geq \mathbb{P}[\tau_{\normalfont\textrm{match}} > \tau_k + n/k] \geq \mathbb{P}[\tau_{\normalfont\textrm{match}} > \tau_k + \gamma n/k]$.
        \end{proof}

        We must introduce one more lemma to finish the two-coordinate chain phase. If two arbitrary chain starting from $(\sigma, \widetilde{\sigma})$ reach to $\tau_\textrm{mag}$, we have not only $S_{\tau_\textrm{mag}} = \widetilde{S}_{\tau_\textrm{mag}}$ but also we have a bound for $|U_{\tau_\textrm{mag}} - \widetilde{U}_{\tau_\textrm{mag}}|$.

        \begin{lemma} \label{lem: two-coordinate phase starting point}
            For arbitrary $\sigma, \widetilde{\sigma} \in \Omega$, consider chains $(X_t, \widetilde{X}_t)$ with $X_0 = \sigma$ and $\widetilde{X}_0 = \widetilde{\sigma}$, and suppose the magnetization matching phase is finished according to Theorem \ref{thm: magnetization matching phase}; i.e. $S_{\normalfont \tau_\textrm{mag}} = \widetilde{S}_{\normalfont \tau_\textrm{mag}}$. At that moment,
            \begin{equation*}
                \mathbb{E}_{\sigma, \widetilde{\sigma}}\Big[|U_{\normalfont \tau_\textrm{mag}} - \widetilde{U}_{\normalfont \tau_\textrm{mag}}|\Big] = O(\sqrt{n})
            \end{equation*}
            holds for $\sigma_0 = \sigma$.
        \end{lemma}
        
        \begin{proof}
            Let $u_0$ be the number of $(+1)$ spins from $\sigma_0$.
            \begin{equation*}
                u_0 = \frac{n(1 + S(\sigma_0))}{2}.
            \end{equation*}
            Then we can observe that
            \begin{equation*}
                U_t = M_t(A_0) + \frac{u_0}{2} \qquad \textrm{ and } \qquad \widetilde{U}_t = \widetilde{M}_t(A_0) + \frac{u_0}{2},
            \end{equation*}
            where $A_0 = \{v : \sigma(v) = 1\}$ ($M_t$ is defined at Lemma \ref{lem: partial sum estimate}). The difference of the two satisfies
            \begin{equation*}
                |U_t - \widetilde{U}_t| = |M_t(A_0) - \widetilde{M}_t(A_0)| \leq |M_t(A_0)| + |\widetilde{M}_t(A_0)|.
            \end{equation*}
            Take expectation for $\sigma$ and $\widetilde{\sigma}$. By Lemma \ref{lem: partial sum estimate}, we get the result.
        \end{proof}
        
        \begin{theorem} \label{two-coordinate chain phase} (two-coordinate chain phase) \\
            Suppose $k = o(\sqrt[3]{n})$. Consider $(X_t, \widetilde{X_t})$ which starts from $(\sigma, \widetilde{\sigma})$. Evolve the chain as given in Theorem \ref{thm: magnetization matching phase} and suppose $\normalfont \tau_\textrm{mag} \leq t_1 + \gamma n/k$ for some $\gamma$. Then their two-coordinate chains $U_t$ and $\widetilde{U}_t$ match after additional $ 2\gamma n/k$ time with high probability.
        \end{theorem}
        
        \begin{proof}
            Recall the definition
            \begin{equation*}
                \Xi_1 := \Big\{\sigma : \min\{U(\sigma), u_0 - U(\sigma), V(\sigma), v_0 - V(\sigma)\}\geq \frac{n}{16} \Big\}
            \end{equation*}
            and let $H(t) := \{(X_t, \widetilde{X}_t) \in \Xi_1 \times \Xi_1\}$. The coupling that we suggested makes $(X_t, \widetilde{X}_t)$ remain in $H(t)$ with high probability. Define
            \begin{equation*}
                D := \bigcup_{t\in[t(\gamma), t(3\gamma)]} \{|M_t(A_0)| \geq n/32\} \qquad \textrm{and} \qquad \delta := \sum_{t \in [t(\gamma), t(3\gamma)]} \mathbbm{1}_{M_t(A_0) > n/64}.
            \end{equation*}
            Since $|M_{t+1}(A_0) - M_t(A_0)| \leq k$, if $|M_{t_0}(A_0)| > n/32$ then $|M_t(A_0)| > n/64$ for all $t$ in any interval of length $\frac{n}{64k}$ containing $t_0$. Therefore $D \subset \{\delta > n/64k\}$ and
            \begin{equation*}
                \mathbb{P}_{\sigma, \widetilde{\sigma}}(D) \leq \mathbb{P}_{\sigma, \widetilde{\sigma}}(\delta > n/64k) \leq \frac{ck\mathbb{E}_{\sigma, \widetilde{\sigma}}[\delta]}{n}.
            \end{equation*}
            Since $X_t$ and $\widetilde{X}_t$ has finished the magnetization phase, from Lemma \ref{lem: partial sum estimate}, provided that $t \geq t(\gamma)$, we have
            \begin{equation*}
                \mathbb{P}_{\sigma, \widetilde{\sigma}}(|M_t(A_0)| > n/64) \leq \frac{64\mathbb{E}_{\sigma, \widetilde{\sigma}}[M_t(A_0)]}{n}= O(n^{-1}),
            \end{equation*}
            thus $\mathbb{E}_{\sigma, \widetilde{\sigma}}[\delta] = O(\gamma)$. Therefore, $\mathbb{P}_{\sigma, \widetilde{\sigma}}(D) = O(\gamma k/ n)$.
            Similar procedure for $\widetilde{M}_t(A_0)$ also gives $\mathbb{P}_{\sigma, \widetilde{\sigma}}(\widetilde{D}) = O(\gamma k/n)$. Suppose $U_t \leq n/16$. Then $u_0 - U_t \geq 3n/16$ as we are assuming $u_0 \geq n/4$. This implies
            \begin{equation*}
                |M_t(A_0)| = |U_t - (u_0 - U_t)| \geq (u_0 - U_t) - U_t \geq \frac{n}{8}.
            \end{equation*}
            Similarly we can get the same result from $(u_0 - U_t) \leq n/16$. This argument is also applicable for $V_t$ and $v_0 - V_t$. Therefore,
            \begin{equation*}
                H(t)^c \subset \{|M_t(A_0)| \geq n/16 \} \cup \{|\widetilde{M}_t(A_0)| \geq n/16\}
            \end{equation*}
            and finally,
            \begin{equation*}
                \mathbb{P}_{\sigma, \widetilde{\sigma}} \Big(\bigcup_{t \in [t(\gamma), t(3\gamma)]} H(t)^c\Big) \leq \mathbb{P}_{\sigma, \widetilde{\sigma}}(D) + \mathbb{P}_{\sigma, \widetilde{\sigma}}(\widetilde{D}) = O(\gamma k / n).
            \end{equation*}
            Now consider the event
            \begin{equation*}
                H = \bigcap_{t \in [t(\gamma), t(3\gamma)]} H(t).
            \end{equation*}
            Combining Lemma \ref{lem: supermartingale lemma} and \ref{lem: two-coordinate chain closes} gives
            \begin{equation*}
                \mathbb{P}_{\sigma, \widetilde{\sigma}}\big[\tau_k > t(2\gamma) |X_{t(\gamma)}, \widetilde{X}_{t(\gamma)}\big] \leq \frac{c|R_{t(\gamma)}|}{\sqrt{n\gamma}},
            \end{equation*}
            taking expectation over $X_{t(\gamma)}$ and $\widetilde{X}_{t(\gamma)}$ with Lemma \ref{lem: two-coordinate phase starting point},
            \begin{equation*}
                \mathbb{P}_{\sigma, \widetilde{\sigma}}\big[\tau_k > t(2\gamma)\big] \leq O(\gamma^{-1/2}).
            \end{equation*}
            When the event $\tau_k \leq t_3 + \gamma n/k$ happens, applying the last coupling with Theorem \ref{lem: tau_k to tau_match} yields
            \begin{equation*}
                \mathbb{P}_{\sigma, \widetilde{\sigma}}\big[\tau_\textrm{match} > \tau_k + \gamma n / k \big] \leq O(\gamma^{-1}).
            \end{equation*}
            To sum up, 
            \begin{equation*}
                \mathbb{P}_{\sigma, \widetilde{\sigma}}\big[\tau_\textrm{match} > t(3\gamma)) \leq O(\gamma^{-1/2}\big].
            \end{equation*}
        \end{proof}

        \begin{proof} (of the theorem \ref{thm: main1 upper bound})\\
            From Lemma \ref{lem: good starting points} and Lemma \ref{lem: distance isometry}, we have
            \begin{equation*}
                \begin{split}
                    d(\theta_0 n/k + t) &\leq \max_{\sigma_0 \in \Omega_0} \big\|\mathbb{P}_{\sigma_0}(X_t \in \cdot) - \mu\big\|_{TV} + O(n^{-1}) \\
                    &= \max_{(u, v) \in \Lambda_0} \big\|\mathbb{P}_{(u, v)}((U_t, V_t) \in \cdot) - \pi \big\|_{TV} + O(n^{-1}).
                \end{split}
            \end{equation*}
            
            Start from two chains $(X_t, \widetilde{X}_t)$. Define an event $G := \{ \tau_\textrm{mag} < t(\gamma)\}$, then thanks to the theorem \ref{thm: magnetization matching phase},
            \begin{equation} \label{eq: 3-5}
                \mathbb{P}_{\sigma, \widetilde{\sigma}}(G^c) = O(\gamma^{-1/2}).
            \end{equation}
            From Theorem \ref{two-coordinate chain phase}, we also have
            \begin{equation} \label{eq: 3-6}
                \mathbb{P}_{\sigma, \widetilde{\sigma}}(H^c | G) = O(\gamma k / n)
            \end{equation}
            and
            \begin{equation} \label{eq: 3-7}
                \mathbb{P}_{\sigma, \widetilde{\sigma}}(\tau_\textrm{match} > t(3\gamma) | H, G) \leq O(\gamma^{-1/2}).
            \end{equation}
            Finally, when the event $G$, $H$, $\{\tau_\textrm{match} < t(3\gamma)\}$ happen, Equation \eqref{eq: 3-5}, \eqref{eq: 3-6}, and \eqref{eq: 3-7} give
            \begin{equation*}
                d(t_4) \leq \mathbb{P}_{\sigma, \widetilde{\sigma}}[\tau_\textrm{match} > t_4] + O(n^{-1}) \leq O(\gamma^{-1/2}) + O(\gamma k/n) + O(n^{-1}) \to 0.
            \end{equation*}
            This is possible because we take limit $n \to \infty$ first and then $\gamma \to \infty$.
        \end{proof}
    
\subsection{Mixing time Lower bound for $\beta < 1$}

    \begin{theorem} \label{thm: main1 lower bound}
        Suppose $k = o(\sqrt{n})$. Then
        \begin{equation*}
            \lim_{\gamma \to \infty} \liminf_{n\to\infty} d_n \Big(\frac{n\log n}{2k(1-\beta)} - \frac{\gamma n}{k} \Big) = 1.
        \end{equation*}
    \end{theorem}

    \begin{proof}
        It is enough to display a suitable lower bound on the distance between $S_t$ and the stationary distribution, as $(S_t)$ is a projection of $(X_t)$. To be more precise, we consider an intermediate sequence $\{Y_i\}$ of $\{X_t\}$ and then suggest a lower bound for $S(Y_i)$. \\
        Let $i' := i \textrm{(mod k)}$ and $k^+, k^-$ be the number of unavailable vertices at each time $i$. We have $k^+ + k^- = i'$ and $k^+, k^- \geq 0$. The drift of $S(Y_i)$ becomes
        \begin{equation*} 
            \mathbb{E}_{s_0}[S(Y_{i+1}) - S(Y_i) | S(Y_i) = s, k^\pm] = \frac{1}{n}\left\{ \frac{n(1-s) - 2k^-}{n-i'}p_+\left(s + \frac{1}{n}\right) - \frac{n(1+s) - 2k^+}{n-i'}p_-\left(s - \frac{1}{n}\right)\right\}.
        \end{equation*}
        If $s \geq 0$, then the equation is equal to
        \begin{equation*}
            \mathbb{E}_{s_0}[S(Y_{i+1}) - S(Y_i) | S(Y_i) = s] \geq \frac{1}{2n(n-i')}\left\{-2ns - 4k + 2n\tanh{\beta s} - O(1)\right\}.
        \end{equation*}
        Since
        \begin{equation*}
            \frac{1}{n(n-i')} = \frac{1}{n^2} + O\left(\frac{k}{n^3}\right),
        \end{equation*}
        combining with the Taylor expansion of $\tanh$, we have
        \begin{equation*}
            \mathbb{E}_{s_0}[S(Y_{i+1}) - S(Y_i) | S(Y_i) = s] \geq -\frac{(1-\beta)s}{n} - \frac{s^3}{2n} - O\left(\frac{k}{n^2}\right).
        \end{equation*}
        From the symmetry of magnetization or direct calculation, we can check that the equation is true for $s \leq 0$ as well. Hence, when $|S(Y_i)| \geq 2/n$,
        \begin{equation*}
            \mathbb{E}_{s_0}\left[|S(Y_{i+1})| \big| S(Y_i)\right] \geq \left(1 - \frac{1-\beta}{n}\right)|S(Y_i)| - \frac{|S(Y_i)|^3}{2n} - O\left(\frac{k}{n^2}\right).
        \end{equation*}
        This is clearly true when $S(Y_i) = 0$ or $|S(Y_i)| = 1/n$ as well. Define $\eta := 1 - \frac{1-\beta}{n}$ and $Z_i = |S(Y_i)|\eta^{-i}$. For large enough $n$ the inequality becomes
        \begin{equation*}
            \mathbb{E}_{Z_0}[Z_{i+1}] \geq Z_i - \frac{\eta^{-i}[|S(Y_i)|^3 + O(k/n)]}{n}.
        \end{equation*}
        As $|S(Y_i)| \leq 1$,
        \begin{equation} \label{eq: 3-8}
            \mathbb{E}_{Z_0}[Z_i - Z_{i+1}] \leq \frac{\eta^{-i}[|S(Y_i)|^2 + O(k/n)]}{n}.
        \end{equation}
        Now, we want to estimate $\mathbb{E}_{s_0}[|S(Y_i)|^2]$. First of all,
        \begin{equation} \label{eq: 3-9}
            \left(\mathbb{E}_{s_0}[S(Y_i)]\right)^2 = \left| \mathbb{E}_{s_0}[S(Y_{i-i'})] + \mathbb{E}_{s_0}[S(Y_i) - S(Y_{i-i'})]\right|^2
        \end{equation}
        The absolute value of the first term of Equation \eqref{eq: 3-9} is bounded by $2|s_0|(1-k(1-\beta)/n)^{(i-i')/k} \leq 2|s_0|\eta^{i-i'}$. The second term of Equation \eqref{eq: 3-9} satisfies
        \begin{equation*}
            |\mathbb{E}_{s_0}[S(Y_i) - S(Y_{i-i'})]| = |\mathbb{E}_{s_0}[\mathbb{E}_{s_0}[S(Y_i) - S(Y_{i-i'})|Y_{i-i'}]]| \leq \mathbb{E}_{s_0}\left[\frac{c_1k}{n}|S(Y_{i-i'})| + \frac{c_2k^2}{n^2}\right],
        \end{equation*}
        which is bounded by $c_3|s_0|k\eta^{i-i'}/n + O(k^2/n^2)$. For the variance,
        $\textrm{Var}[S(Y_i)] = \textrm{Var}[S(Y_{i-i'})] + \textrm{Var}[S(Y_i)|S(Y_{i-i'})] \leq O(1/n) + O(k/n^2) = O(1/n)$. Hence
        \begin{equation*}
            \begin{split}
                \mathbb{E}_{s_0}[|S(Y_i)|^2] &= \left(\mathbb{E}_{s_0}[S(Y_i)]\right)^2 + \textrm{Var}_{s_0}[S(Y_i)] \leq \left[c_4|s_0|\eta^{i-i'} + O\left(\frac{k^2}{n^2}\right)\right]^2 + O\left(\frac{1}{n}\right) \\
                &\leq c_5|s_0|^2\eta^{2(i-k)} + \frac{c_6k^2|s_0|\eta^{i-k}}{n^2} + O\left(\frac{k^4}{n^4} + \frac{1}{n} \right).
            \end{split}
        \end{equation*}
        Taking expectation on Equation \eqref{eq: 3-8} becomes
        \begin{equation} \label{eq: 3-10}
            \mathbb{E}_{Z_0}[Z_i - Z_{i+1}] \leq \frac{c_5|s_0|^2\eta^{i-2k}}{n} + \frac{c_6k^2|s_0|\eta^{-k}}{n^3} + \eta^{-i}O\left(\frac{k}{n^2}\right).
        \end{equation}
        Let $i^* = n\log n / 2(1-\beta) - \gamma n/(1-\beta)$ and sum up Equation \eqref{eq: 3-10} for all $i = 0, 1, ..., i^*$.
        \begin{equation*}
            s_0 - \mathbb{E}_{s_0}[Z_{i^*}] \leq \frac{c_5|s_0|^2\eta^{-2k}}{n(1-\eta)} + \frac{c_7k^2|s_0|\eta^{-k}\log n}{n^2} + O\left(\frac{k}{\sqrt{n}}\right).
        \end{equation*}
        Because $\eta^{-i^*} \leq n^{1/2}$, $\eta^{-k} \to 1$ as $n \to \infty$, two terms vanish as $n \to \infty$. For large $n$, if we choose $s_0 < \frac{1-\beta}{3c_5}$ then the right side of the equation is less than $s_0/2$. If so,
        \begin{equation*}
            \mathbb{E}_{s_0}[|S(Y_{i^*})] \geq \frac{s_0\eta^{i^*}}{2} \geq \frac{s_0e^\gamma}{2\sqrt{n}} =: A
        \end{equation*}
        As we discussed before, since $\textrm{Var}_{s_0}[S(Y_{i^*})] = O(n^{-1})$,
        \begin{equation*}
            \mathbb{P}_{s_0}(|S(Y_{i^*})| < A/2) \leq \mathbb{P}_{s_0}(|S(Y_{i^*}) - \mathbb{E}_{s_0}S(Y_{i^*})| > A/2) \leq \frac{4\textrm{Var}_{s_0}[S(Y_{t^*})]}{A^2} \leq O(e^{-2\gamma}s^{-2}_0).
        \end{equation*}
        On the other point of view, $\mathbb{E}_\pi[S] = 0$ and $\textrm{Var}_\pi[S] = O(n^{-1})$. Therefore
        \begin{equation*}
            \mathbb{P}_\pi(|S| > A/2) \leq \frac{4\textrm{Var}_\pi[S]}{A^2} = O(e^{-2\gamma}s^{-2}_0).
        \end{equation*}
        Finally, let $\mathcal{A} := [-A/2, A/2]$. Then
        \begin{equation*}
            \|\mathbb{P}_{s_0}(S_{i^*/k} \in \cdot) - \pi \|_\textrm{TV} \geq \pi(\mathcal{A}) - \mathbb{P}_{s_0}(|S_{i^*/k}| \leq A/2) \geq 1 - O(e^{-2\gamma} s^2_0). 
        \end{equation*}
        Taking $\gamma \to \infty$ ensures the total variation distance goes to $1$.
        
    \end{proof}

    \begin{remark*}
        Theorem \ref{thm: main1 lower bound} is only available when $k = o(\sqrt{n})$. Without this condition we can derive another lower bound under $k=o(n)$. This work can be done with generalized Wilson's lemma, which is from \cite{NN18}.
        \begin{equation*} \label{eq: 3-11}
            \lim_{\gamma \to \infty} \liminf_{n\to\infty} d_n \Big(\frac{n\log n}{2k(1-\beta)} - \frac{n\log k}{2k(1-\beta)} - \frac{\gamma n}{k} \Big) = 1.
        \end{equation*}
        However, $\log k = o(\log n)$ condition is required for cutoff when Equation \ref{eq: 3-11} is combined with the upper bound result.
    \end{remark*}
        
\section{Main results: critical temperature regime}

        In this section we show that the mixing time under the randomized systematic scan dynamics is of order $n^{3/2}/k$ when $\beta = 1$. Assume $\beta = 1$ throughout this section.
        
    \subsection{Mixing time Upper bound for $\beta = 1$}    
        \begin{theorem} \label{thm: main2 upper bound}
            Suppose $\beta=1$. If $k = o(\sqrt[4]{n})$ then $t_{mix} = O(n^{3/2}/k)$.
        \end{theorem}
        
        \begin{proof}
            The proof consists of two steps. With some positive probability, for any given two chains, their magnetization values eventually agree in time of order $n^{3/2}/k$. After two magnetizations coalesce, two chains can be matched by using the same matching method used in the case of $\beta < 1$. Markov inequality and Lemma \ref{lem: magnetization match to exact match} ensures the exact match can be done in $O(n\log n / k)$ time for the second step. It is enough to show that the magnetizations of two arbitrary chains can be matched in time of order $n^{3/2}/k$. Recall Theorem \ref{lem: apriori estimate}; for any chain $X_t$ we have
            \begin{equation*}
                \Big| \mathbb{E}_\sigma[S_1] - \left(1-\frac{k}{n}\right)S_0 - \frac{k}{n}\tanh \beta S_0 \Big| \leq \frac{2k}{n}\tanh\beta\frac{2k}{n} \leq \frac{4k^2}{n^2}.
            \end{equation*}
            Define $\tau_0 = \min \{t \geq 0 : |S_t| \leq 2\sqrt[3]{k/n}\}$. For any $t < \tau_0$, as $|S_{t+1}-S_t| \leq 2k/n <2\sqrt[3]{k/n} $, the sign of both $S_t$ and $S_{t+1}$ are the same. In this case, we can apply the absolute value on the magnetization:
            \begin{equation*}
                \mathbb{E}_{\sigma}\big[|S_{t+1}| \big| S_t\big] \leq \big(1-\frac{k}{n}\big)|S_0| + \frac{k}{n}\tanh |S_0| + \frac{4k^2}{n^2}.
            \end{equation*}
            Let $\xi_t = \mathbb{E}\big[|S_t| \big| \mathbbm{1}\{\tau_0 > t\}\big]$. Then the above becomes
            \begin{equation*}
                \xi_{t+1} \leq \Big(1-\frac{k}{n}\Big)\xi_t + \frac{k}{n}\tanh\xi_t + \frac{4k^2}{n^2}.
            \end{equation*}
            Thus if $\xi_t > \epsilon$, there exist $c_\epsilon > 0$ which satisfies
            \begin{equation*}
                \xi_{t+1} - \xi_t \leq -\frac{k c_\epsilon}{n},
            \end{equation*}
            provided $k=o(n)$. Therefore in $t_* = O(n/k)$ time we can get $\xi_t \leq 1/4$, for all $t \geq t_*$. \\
            Taylor series of $\tanh$ gives
            \begin{equation*}
                \xi_{t+1} \leq \xi_t - \frac{k\xi^3_t}{4n} + \frac{4k^2}{n^2}
            \end{equation*}
            for $t \geq t_*$. Therefore for some large enough $n$, $\xi_t$ is a decreasing sequence. It is enough to consider the cases in which $n$ is large enough. Consider a decreasing sequence $b_i = (1/4)2^{-i}$ and define $u_i := \min \{t>t_*: \xi_t \leq b_i\}$. We can check that $b_{i+1} < \xi_t \leq b_i$ when $u_i \leq t < u_{i+1}$. Now, for any $t \in (u_i, u_{i+1}]$,
            \begin{equation*}
                \xi_{t+1} \leq \xi_t - \frac{kb^3_i}{32n} + O\Big(\frac{k^2}{n^2}\Big)
            \end{equation*}
            holds, and summing up for all $t \in (u_i, u_{i+1}]$ gives
            \begin{equation*}
                u_{i+1} - u_i \leq \frac{16n}{kb^2_i}[1 + O(b^{-3}_i n^{-1}k)].
            \end{equation*}
            
            Now, let $i_0 = \min\{i : b_i \leq n^{-1/4}\}$. For any $i<i_0$, we have $O(b^{-3}_i n^{-1}k) = o(1)$ with the assumption $k = o(\sqrt[4]{n})$. Therefore, for large $n$ and $0 \leq i < i_0$,
            \begin{equation} \label{eq: 4-1}
                u_{i+1} - u_i \leq \frac{32n}{kb^2_i}.
            \end{equation}
            Summing up Equation \eqref{eq: 4-1} for all $i = 0, ..., i_0$ becomes
            \begin{equation*}
                u_{i_0} - u_0 \leq \sum^{i_0 - 1}_{i=0} \frac{32n}{kb^2_i} \leq \frac{cn}{kb^2_{i_0 - 1}} = O(n^{3/2}k^{-1}).
            \end{equation*}
            Since $u_0 = t_* = O(n/k)$, finally we derive
            \begin{equation*}
                u_{i_0} \leq O\Big(\frac{n^{3/2}}{k}\Big) + O\Big(\frac{n}{k}\Big) = O\Big(\frac{n^{3/2}}{k}\Big).
            \end{equation*}
            The $O(n/k)$ term comes from that $u_0 = t_* = O(n/k)$. Now let $r_n = c_1n^{3/2}k^{-1}$, then
            \begin{equation} \label{eq: 4-2}
                \mathbb{E}_{\sigma}\big[|S_{r_n}|\mathbbm{1}\{\tau_0 > r_n\}\big] = O(n^{-1/4}).
            \end{equation}
            
            Appealing to Lemma \ref{lem: supermartingale lemma} on $|S_t|$ gives
            \begin{equation} \label{eq: 4-3}
                \mathbb{P}_{\sigma}\Big(\tau_0 > r_n + \frac{\gamma n^{3/2}}{k} \Big| X_{r_n}\Big) \leq \frac{C|S_{r_n}|}{\sqrt{\frac{k}{n^2}}\sqrt{\frac{\gamma n^{3/2}}{k}}} = \frac{C|S_{r_n}|}{\sqrt{\gamma}n^{1/4}}
            \end{equation}
            Therefore, multiply $\mathbbm{1}\{\tau_0 > r_n\}$ on Equation \eqref{eq: 4-3} and taking expectation with \label{eq: 4-2} gives
            \begin{equation*}
                \mathbb{P}_{\sigma}\Big(\tau_0 > r_n + \frac{\gamma n^{3/2}}{k}\Big) = O(\gamma^{-1/2}).
            \end{equation*}
            Since $r_n = O(n^{3/2}k^{-1})$, we can reach $\tau_0$ in $O(n^{3/2}k^{-1})$ time with high probability.
            
            Now consider two chains $\{X_t\}$ and $\{\widetilde{X}_t\}$. Without loss of generality assume $\tau_0$ of $\{X_t\}$ is greater than that of $\{\widetilde{X}_t\}$. Define $\tau_\textrm{cross}$ with respect to their intermediates as in Lemma \ref{lem: crossing is almost coupling}. We know that
            \begin{equation*}
                \mathbb{P}_{\sigma, \widetilde{\sigma}}\left(\tau_0 > \frac{\gamma_1 n^{3/2}}{k}\right) \leq O(\gamma_1^{-1/2}).
            \end{equation*}
            In case of $\tau_\textrm{cross} < k\tau_0$ then by Lemma \ref{lem: crossing is almost coupling} we can couple the magnetizations with a positive probability. This proves $\mathbb{P}[\tau_\textrm{mag} \leq \tau_0]$ is uniformly bounded above from $0$.  On the other hand, if $\tau_\textrm{cross} > \tau_0$, that means both $S_{\tau_0}$ and $\widetilde{S}_{\tau_0}$ are in $[-2\sqrt[3]{k/n}, 2\sqrt[3]{k/n}]$ with high probability. Define
            \begin{equation*}
                Z_t := S_t - \widetilde{S}_t + \frac{4k}{n}
            \end{equation*}
            and run two chains $X_t$ and $\widetilde{X}_t$ independently from $\tau_0$. Set a stopping time $\tau_1 = \min\{t\geq \tau_0: Z_t \leq \frac{4k}{n}\}$. From the assumption $\tau_\textrm{cross} > \tau_0$ and $\tau_\textrm{cross}$ for $X_t$ is greater, we have $Z_{\tau_0} > \frac{4k}{n}$. Apply Lemma \ref{lem: supermartingale lemma} on $Z_t$ gives
            \begin{equation*}
                \mathbb{P}_{\sigma, \widetilde{\sigma}}\left[\tau_1 > \frac{n^{3/2}\gamma}{k} + \tau_0 \big| Z_{\tau_0}\right] \lesssim \frac{Z_{\tau_0}}{\sqrt{\frac{k}{n^2}}\sqrt{\frac{n^{3/2}\gamma}{k}}} \lesssim \frac{k^{1/3}}{n^{1/12}\gamma^{1/2}} \lesssim \gamma^{-1/2}.
            \end{equation*}

            However, $k\tau_1 \leq k\tau_\textrm{cross}$ by definition, so we can capture the time $\tau_\textrm{cross}$ and follow the coupling from Lemma \ref{lem: crossing is almost coupling}. In sum up, $\mathbb{P}[\tau_\textrm{mag} > \frac{\gamma n^{3/2}}{k}] < c < 1$. This is enough to show $t_\textrm{mix} = O(n^{3/2}/k)$.
        \end{proof}
        \medskip
        
\subsection{Mixing time Lower bound for $\beta = 1$}
        \begin{theorem} \label{thm: main2 lower bound}
            There exists a constant $c>0$ such that $t_{mix} \geq cn^{3/2}/k$.
        \end{theorem}
        
        \begin{proof}
            Since the magnetization chain $S_t$ depends on the original chain, it is enough to prove for the lower bound of the mixing time of $S_t$.
            If we denote $S$ to be a magnetization in equilibrium, then the sequence $n^{1/4}S$ converges to a non-trivial limit law as $n \to \infty$, according to \cite{SG73} and \cite{Ellis66}(See Theorem V.9.5). Therefore, pick $A>0$ such that
            \begin{equation*}
                \mu\Big(|S| \leq An^{-1/4}\Big) \geq 3/4,
            \end{equation*}
            and also let $s_0 := 2An^{-1/4}$. Now define $\widetilde{S_t}$ as a chain with the same transition probability as $S_t$, except $s = s_0$.
            
            Define $Z_t = \widetilde{S}_0 - \widetilde{S}_{t \wedge  \tau}$, where $\tau := \min\{t\geq 0 : \widetilde{S}_t \leq An^{-1/4}\}$. When $An^{-1/4} < \widetilde{S}_t = s' < s_0$, the conditional distribution of $\widetilde{S}_{t+1}$ is the same to that of $S_{t+1}$ in case of $S_t = s'$. Therefore,
            \begin{equation*}
                \mathbb{E}_{s_0}\Big[\widetilde{S}_{t+1} | \widetilde{S}_t = s \Big] = \mathbb{E}_{s_0}\Big[S_{t+1} | S_t = s \Big] \geq s - \frac{c_0ks^3}{n}
            \end{equation*}
            holds for some constant $c_0$. From this, in terms of $Z_t$,
            \begin{equation*}
                \mathbb{E}[Z_{t+1}|\mathcal{F}_t] \leq Z_t + \frac{c_0k}{n} S^3_t.
            \end{equation*}
            Now, consider the second moment of $Z_{t+1}$;
            \begin{equation} \label{eq: 4-4}
                \mathbb{E}_{s_0}\Big[Z^2_{t+1} | \mathcal{F}_t \Big] = \textrm{Var}(Z_{t+1} | \mathcal{F}_t) + \Big(\mathbb{E}_{s_0}[Z_{t+1} | \mathcal{F}_t]\Big)^2.
            \end{equation}
            Lemma \ref{lem: variance estimate no.1} ensures
            \begin{equation*}
                \textrm{Var}(Z_{t+1} | \mathcal{F}_t) \leq O\Big(\frac{k}{n^2}\Big).
            \end{equation*}
            The last term of Equation \eqref{eq: 4-4}, for $t < \tau$, there is another constant $c_1 = c_1(A)$ such that
            \begin{equation*}
                \mathbb{E}^2_{s_0}[Z_{t+1} | \mathcal{F}_t] \leq Z^2_t + 2\frac{c_0 k}{n}Z_t\widetilde{S}^3_t + \frac{c^2_0k^2\widetilde{S}^6_t}{n^2} \leq Z^2_t + \frac{c_1k}{n^2}.
            \end{equation*}
            In conclusion, we can get
            \begin{equation*}
                \mathbb{E}_{s_0}\Big[Z^2_{t+1} - Z^2_t | \mathcal{F}_t\Big] \leq \frac{c_Ak}{n^2}.
            \end{equation*}
            This means $\mathbb{E}_{s_0}[Z^2_t] \leq \frac{c_Akt}{n^2}$, which leads to
            \begin{equation*}
                c_An^{-2}t \leq \mathbb{E}_{s_0}[Z^2_t] \leq \mathbb{E}_{s_0}[Z^2_t\mathbbm{1}_{\{\tau \leq t\}}] \leq \frac{A^2}{\sqrt{n}}\mathbb{P}_{s_0}(\tau \leq t)
            \end{equation*}
            Take $t = (A^2/4c_A)n^{3/2}/k$. Then
            \begin{equation*}
                \mathbb{P}_{s_0}\Big(S_t \leq An^{-1/4}\Big) \leq \frac{1}{4},
            \end{equation*}
            which proves $d(cn^{3/2}/k) \geq 1/2$.
        \end{proof}

\section{Main results: low temperature regime}
    We assume $\beta > 1$ throughout this section. It is well known that the mixing time on low temperature regime is exponential, as shown by the Cheeger constant in \cite{GWL66}. To this end \cite{LLP10} suggested the restricted Glauber dynamics holds only for $\Omega^+ := \{\sigma \in \Omega : S(\sigma) \geq 0\}$. This can be also generalized for the randomized scan dynamics. The \textit{restricted randomized scan dynamics} for low temperature is the following update scheme: for any given $\sigma \in \Omega^+$, generate a candidate $\sigma'$ by $k$ site of updates from $\sigma$ with the usual randomized systematic scan dynamics. If $S(\sigma') \geq 0$ we accept $\sigma'$ for the next state; if $S(\sigma') < 0$ then we accept $-\sigma'$, the state whose spins are all reversed from $\sigma'$, for the next state. Analyzing corresponding magnetization chain takes an important role, so we denote $S^+_t := S(X_t)$ with the plus sign to emphasize we are working under the restricted Glauber dynamics.

\subsection{Mixing time Upper bound for $\beta > 1$}
    \begin{theorem} \label{thm: main3 upper bound}
        Under the restricted Glauber dynamics, $t_\textrm{mix} \leq \frac{c_1n\log n }{ k}$ for some constant $c_1 = c_1(\beta) > 0$.
    \end{theorem}

    \begin{lemma} \label{lem: hitting time estimate from above}
        Suppose $k = o(\sqrt{n})$. Define $s^*$ to be the unique positive solution of $\tanh(\beta s) = s$, and for constant $\alpha$ let
        \begin{equation*}
            \tau^* = \tau^*(\alpha) := \inf\Big\{ t \geq 0 : S^+_t \leq s^* + \frac{\alpha}{\sqrt{n}}\Big\}.
        \end{equation*}
        Then, for some suitable $c = c(\alpha, \beta) > 0$,
        \begin{equation*}
            \lim_{n \to \infty}\mathbb{P}_{\sigma}(\tau^* > cn\log n / k) = 0.
        \end{equation*}
    \end{lemma}
    \medskip
    
    \begin{proof}
        Start from Theorem \ref{lem: apriori estimate}. In case of $S^+_t > 2k/n$,
        \begin{equation*}
            \mathbb{E}_\sigma[S^+_{t+1} - S^+_t | S^+_t = s] \leq \frac{k}{n}(\tanh(\beta S^+)-S^+) + \frac{4k^2}{n^2}.
        \end{equation*}
        Define $\gamma^* := \beta \cosh^{-2}(\beta s^*)$. By the mean value theorem, for $y > 0$,
        \begin{equation*}
            \tanh(\beta(s^* + y)) - \tanh(\beta s^*) = \frac{\beta}{\cosh^2(s^* + \tilde{y})}y \leq \gamma y.
        \end{equation*}
        Therefore for $y \geq 0$, $\tanh(\beta(s^* + y)) \leq s^* + \gamma^*y$. In conclusion, provided that $S^+_t > (2k/n \vee s^*)$,
        \begin{equation*}
            \mathbb{E}_\sigma[S^+_{t+1} - S^+_t | S^+_t = s] \leq -\frac{k(1-\gamma^*)}{n}(s-s^*) + \frac{4k^2}{n^2}.
        \end{equation*}
        Define
        \begin{equation*}
            Y_t := \Big[1 - \frac{k(1-\gamma^*)}{n}\Big]^{-t}\Big(S^+_t - s^* - \frac{4k}{n(1-\gamma^*)}\Big).
        \end{equation*}
        Take large $n$ so that $(2k/n \vee s^*) = s^*$, then $Y_t$ becomes non-negative supermartingale for $t < \tau^*$ since $s^*$ only depends on $\beta$. From the optional stopping lemma,
        \begin{equation*}
            1 \geq \mathbb{E}_\sigma[Y_{\tau^* \wedge t}] \geq \Big[1 - \frac{k(1-\gamma^*)}{n}\Big]^{-t}\Big(\frac{\alpha}{\sqrt{n}} - \frac{4k}{n(1-\gamma^*)}\Big)\mathbb{P}_\sigma(\tau^* > t).
        \end{equation*}
        Hence
        \begin{equation*}
            \mathbb{P}_\sigma(\tau^* > t) \leq \Big(\frac{\alpha}{\sqrt{n}} - \frac{4k}{n(1-\gamma^*)}\Big)^{-1}\Big[1 - \frac{k(1-\gamma^*)}{n}\Big]^t.
        \end{equation*}
        Plugging in $t = cn\log n / k$ finishes the proof, as $k/n = o(1/\sqrt{n})$.
    \end{proof}
    \medskip
    
    \begin{lemma} \label{lem: hitting time estimate from below}
        Define $s^*$ in the same way as Lemma \ref{lem: hitting time estimate from above}. For any $\alpha>0$, define
        \begin{equation*}
            \tau_* = \tau_*(\alpha) := \min \Big\{t \geq 0 : S^+_t \geq s^* + \frac{\alpha}{\sqrt{n}} \Big\}.
        \end{equation*}
        If $k = o(\sqrt{n})$, then $\mathbb{E}_0[\tau_*] = O(n \log n / k)$.
    \end{lemma}

    The proof of Lemma \ref{lem: hitting time estimate from below} is suggested on the appendix. With Lemma \ref{lem: hitting time estimate from above} and Lemma \ref{lem: hitting time estimate from below}, the upper bound can be proven.

    \begin{proof}(Theorem \ref{thm: main3 upper bound})
        First of all, for any arbitrary configurations $\sigma$ and $\widetilde{\sigma}$ in $\Omega^+$ which satisfies $S(\sigma) = S(\widetilde{\sigma})$, we can show that two chains start from these two configurations can be exactly matched with high probability in $O(n\log n / k)$ times. This can be done by applying Lemma \ref{lem: magnetization match to exact match} in a slightly different way. From this fact, it is enough to show that there exists $C>0$ that for any two states $\sigma$ and $\widetilde{\sigma}$,
        \begin{equation}
            \lim_{C \to \infty} \limsup_{n \to \infty}\mathbb{P}_{\sigma, \widetilde{\sigma}}(\tau_\textrm{mag} > Cn\log n / k) < 1,
        \end{equation}
        where $\tau_\textrm{mag}$ is the first time $t$ with $S^+_t = \widetilde{S}^+_t$. \\
        Although the monotone coupling for $X_t$ we have discussed does not make sense under the restricted dynamics, there is still a coupling for two magnetization chains $(S^+_t, \widetilde{S}^+_t)$ which preserves monotonicity between magnetizations; i.e. $(S^+_t - \widetilde{S}^+_t)(S^+_{t+1} - \widetilde{S}^+_{t+1}) \geq 0$. Due to this monotonicity, it is enough to consider two starting points $0$ and $1$(odd $n$ case can be also done in an analogous way). Consider two chains, denoted as $S^+_T$ and $S^+_B$, whose starting positions are $1$ and $0$ respectively. Let $\mu^+$ be the stationary distribution of the restricted magnetization chain and $S^+_\mu$ be a stationary copy of the restricted magnetization, i.e. whose initial distribution is $\mu^+$. Run three chains $S^+_T, S^+_B$, and $S^+_\mu$ independently at the beginning. For some constants $0 < c_1 \leq c_2$, define stopping times as
        \begin{equation*}
            \begin{split}
                \tau_1 &:= \min \Big\{t \geq 0 : S^+_{T, t} \leq s^* + c_1 n^{-1/2} \Big\}, \\
                \tau_2 &:= \min \Big\{t \geq 0 : S^+_{B, t} \geq s^* + c_2 n^{-1/2} \Big\}.
            \end{split}
        \end{equation*}
        Assume $\tau_1 < \tau_2$. If $S^+_{\mu, \tau_1} \geq s^* + c_1 n^{-1/2}$, for $t \geq \tau_1$ we couple $S^+_\mu$ chain and $S^+_T$ chain, while $S^+_B$ runs independently. On the event $S^+_{\mu, \tau_1} < s^* + c_1 n^{-1/2}$, continue running all three chains independently.
        After reaching the time $\tau_2$, if $S^+_{\mu, \tau_2} \leq s^* + c_2 n^{-1/2}$ couple all three chains monotonically. If $S^+_{\mu, \tau_2} > s^* + c_2 n^{-1/2}$, continue running all three chains independently. The other case, $\tau_1 \geq \tau_2$ can be set similarly. \\
        For another constant $c_3 > 0$, define events $H_1, H_2$ as
        \begin{equation*}
            \begin{split}
                H_1 &:= \{\tau_1 \leq c_3 n \log n / k\} \cap \Big\{S_{\mu, \tau_1} \geq s^* + c_1 n^{-1/2}\Big\}, \\
                H_2 &:= \{\tau_2 \leq c_3 n \log n / k\} \cap \Big\{S_{\mu, \tau_2} \leq s^* + c_2n^{-1/2} \Big\}.
            \end{split}
        \end{equation*}
        Then we have
        \begin{equation*}
            \begin{split}
            \mathbb{P}_{\sigma, \widetilde{\sigma}}(H_1^c) &\leq \mathbb{P}_{\sigma, \widetilde{\sigma}}(\tau_1 > c_3 n \log n / k) + \mu^+(0, s^* + c_1 n^{-1/2}) \\
            \mathbb{P}_{\sigma, \widetilde{\sigma}}(H_2^c) &\leq \mathbb{P}_{\sigma, \widetilde{\sigma}}(\tau_2 > c_3 n \log n / k) + \mu^+(s^* + c_2 n^{-1/2}, 1).
            \end{split}
        \end{equation*}
        Observe that on the event $H_1 \cap H_2$ the chains $S_T$ and $S_B$ cross over each other by $c_3 n \log n / k$, therefore
        \begin{equation*}
            \begin{split}
                \mathbb{P}_{\sigma, \widetilde{\sigma}} &\geq 1 - \mathbb{P}_{\sigma, \widetilde{\sigma}}(\tau_1 > c_3 n \log n / k) - \mathbb{P}_{\sigma, \widetilde{\sigma}}(\tau_2 > c_3 n \log n / k) \\
                &\qquad \qquad - \mu^+((s^* + c_1 n^{-1/2}, s^* + c_2 n^{-1/2})^c).
            \end{split}
        \end{equation*}
        $\mu^+$ satisfies a central limit theorem(\cite{ENR80}), therefore the last term $\mu^+((s^* + c_1 n^{-1/2}, s^* + c_2 n^{-1/2})^c)$ is uniformly 
        bounded away from $1$. Further, from the previous two Lemmas \ref{lem: hitting time estimate from above} and \ref{lem: hitting time estimate from below}, we have
        \begin{equation*}
        \lim_{n \to \infty}\mathbb{P}_{\sigma, \widetilde{\sigma}}(\tau_1 > c_3 n \log n / k) = \lim_{n \to \infty}\mathbb{P}_{\sigma, \widetilde{\sigma}}(\tau_2 > c_3 n \log n / k) = 0.
        \end{equation*}
        Therefore $\mathbb{P}_{\sigma, \widetilde{\sigma}}(H_1 \cap H_2)$ is bounded away from $0$, and this means $S^+_T$ and $S^+_B$ cross by the time $c_3 n \log n / k$. As a final step, define
        \begin{equation*}
            \tau_\textrm{loc} = \min\{t \geq 1 : (S^{+T}_{T, t-1} - S^{+T}_{B, t-1})(S^{+T}_{T, t} - S^{+T}_{B, t}) \leq 0\}
        \end{equation*}
        then $\mathbb{P}_{\sigma, \widetilde{\sigma}}(\tau_\textrm{loc} < c_3 \log n) > \epsilon$ uniformly on $n$. From $k^2 = o(n)$ condition, there is a coupling between $S^{+T}_T$ and $S^{+T}_B$ such that
        \begin{equation*}
            \mathbb{P}_{\sigma, \widetilde{\sigma}}\Big(\tau_\textrm{mag} = \left\lceil \frac{\tau_\textrm{loc}}{k}\right\rceil\Big) > \epsilon > 0,
        \end{equation*}
        where $\epsilon$ is independent to $n$.
    \end{proof}

\subsection{Mixing time Lower bound for $\beta > 1$}
    \begin{theorem} \label{thm: main3 lower bound}
        Under the restricted Glauber dynamics, $t_\textrm{mix} \geq (1/4)n\log n / k$.
    \end{theorem}
    
    \begin{proof}
        The proof can be done in the similar fashion as \cite{LLP10}. Again, define $s^*$ to be the unique positive solution of the equation $\tanh(\beta s) = s$. Let $\{X_t\}$ start from $\mathbbm{1}$, and let $\{\widetilde{X}_t\}$ follows the stationary distribution $\mu^+$. Apply the monotone coupling to , and write $\mathbb{P}_{\mathbbm{1}, \mu^+}$ and $\mathbb{E}_{\mathbbm{1}, \mu^+}$ be the probability measure and expectation under this coupling. Define $\mathcal{B}(\sigma)$ be a set of vertices with a minus spin, and $B(\sigma) := |\mathcal{B}(\sigma)|$.
        From the stationary magnetization result of \cite{ENR80}, for some suitable $0 < c_1 < 1$,
        \begin{equation*}
            \mathbb{P}_{\mathbbm{1}, \mu^+}(B(\widetilde{X}_0) \leq c_1 n) = \mu^+(\{\sigma : B(\sigma) \leq c_1 n\}) = o(1).
        \end{equation*}
        Let $N_t$ be the number of the sites in $\mathcal{B}(\widetilde{X}_0)$ that have not been updated until time $t$. Then,
        \begin{equation*}
            \mathbb{E}_{\mathbbm{1}, \mu^+} [N_t | B(\widetilde{X}_0)] = B(\widetilde{X}_0)\left(1-\frac{k}{n}\right)^t.
        \end{equation*}
        Plugging in $t^* := (1/4)n\log n / k$ gives
        \begin{equation*}
            \mathbb{E}_{\mathbbm{1}, \mu^+} [N_{t^*} | B(\widetilde{X}_0)] \geq c_2B(\widetilde{X}_0)n^{-1/4}.
        \end{equation*}
        If we consider $N_t$ as a sum of indicators showing the update status on each vertices, for any $v, w \in \mathcal{B}(\widetilde{X}_0)$,
        \begin{equation*}
            \mathbb{E}_{\mathbbm{1}, \mu^+}[I_v I_w] = (1 - \frac{2k}{n})^t \leq (1 - \frac{k}{n})^{2t} = \mathbb{E}_{\mathbbm{1}, \mu^+}[I_v]\mathbb{E}_{\mathbbm{1}, \mu^+}[I_w],
        \end{equation*}
        so indicators are negatively correlated. This ensures $\textrm{Var}_{\mathbbm{1}, \mu^+}(N_t) \leq n$ at all time. Combining with Chebyshev's inequality, on the event $E_1 := \{B(\widetilde{X}_0) > c_1 n\}$, with some $c_3 > 0$,
        \begin{equation*}
            \mathbb{P}_{\mathbbm{1}, \mu^+}[N_{t^*} \leq c_3n^{3/4} | B(\widetilde{X}_0)] = o(1).
        \end{equation*}
        Therefore,
        \begin{equation*}
            \mathbb{P}_{\mathbbm{1}, \mu^+}[N_{t^*} \leq c_3n^{3/4}] \leq \mathbb{P}_{\mathbbm{1}, \mu^+}(E_1^c) + \mathbb{P}_{\mathbbm{1}, \mu^+}(E_1 \cap \{N_{t^*} \leq c_3n^{3/4}\}) = o(1).
        \end{equation*}
        Now, suppose $N_{t^*} \leq c_3n^{3/4}$. In this case we have $S_{t^*} \geq \widetilde{S}_{t^*} + c_4n^{-1/4}$ for some $c_4 > 0$. Pick a small constant $c_5 \in (0, c_4)$ and define $E_2 := S_{t^*} \leq s^* + c_5n^{-1/4}$. In this case,
        \begin{equation*}
            \begin{split}
            \mathbb{P}_{\mathbbm{1}, \mu^+}(E_2) &\leq \mathbb{P}_{\mathbbm{1}, \mu^+}(N_{t^*} > c_3n^{3/4}) + \mathbb{P}_{\mathbbm{1}, \mu^+}(E_2 \cap \{N_{t^*} \leq c_3n^{3/4}\}) \\
            &\leq o(1) + \mathbb{P}_{\mathbbm{1}, \mu^+}(\widetilde{S}_{t^*} \leq s^* +(c_5 - c_4)n^{-1/4}) \\
            &=o(1),
            \end{split}
        \end{equation*}
        by appealing to the central limit theorem at the last equation. Furthermore the theorem gives 
        \begin{equation*}
            \mu^+\Big(\{\sigma : S(\sigma) > s^* + c_5n^{-1/4}\}\Big) = o(1).
        \end{equation*}
        Finally,
        \begin{equation*}
            d(t^*) \geq \mathbb{P}_{\mathbbm{1}, \mu^+}(\{\sigma : S(\sigma) > s^* + c_5n^{-1/4}\}) - \mu^+(\{\sigma : S(\sigma) > s^* + c_5n^{-1/4}\}) = 1 - o(1).
        \end{equation*}
        Therefore we have $t_\textrm{mix}(n) \geq (1/4)n\log n / k$.
    \end{proof}

    \medskip
    
\section*{Acknowledgement}
    We are sincerely grateful to Evita Nestoridi for introducing the topic and helpful discussions, and also to Seung-Yeon Ryoo for the pivotal contribution to the low-temperature regime calculations including the Appendix.
\newpage

\section{Appendix: proof of Lemma \ref{lem: hitting time estimate from below}}
    In this section we suggest a detailed calculation for Lemma \ref{lem: hitting time estimate from below}. Throughout this section we assume $\beta > 1$, $k = o(n)$, and let $s^*$ be the unique positive solution of $\tanh(\beta s) = s$. This relationship can be modified to
    \begin{equation*}
        \beta = \frac{1}{2s^*}\log\Big(\frac{1 + s^*}{1 - s^*}\Big),
    \end{equation*}
    and Taylor expansion on the right side at $s=0$ gives $\beta > 1 + (s^*)^2/3$.
    
    \begin{proposition} \label{prop: auxiliary function f and I}
    Define
        \begin{equation*}
            I(x) := -2\beta x^2 + \frac{\lambda k}{n}x + \Big(\frac{1}{2} + x\Big)\log\Big(\frac{1}{2} + x\Big) + \Big(\frac{1}{2} - x\Big)\log\Big(\frac{1}{2} - x\Big) + \log 2.
        \end{equation*}
    Then there exists $s_0, s_1$ such that $I'(s_0) = I'(s_1) = 0$, $0 < s_0 = O(k/n)$, and $s^*/2 > s_1 = s^*/2 - O(k/n)$.
    \end{proposition}
    
    \begin{proof}
        Since
        \begin{equation*}
            \begin{split}
            I'(x) &= -4\beta x + \log\big(\frac{1}{2} + x\big) - \log\big(\frac{1}{2} - x\big) + \frac{\lambda k}{n} \\
            I''(x) &= -4\beta + \frac{4}{1-4x^2},
            \end{split}
        \end{equation*}
        we have $I'(s^*/2) = I'(0) = \lambda k / n > 0$. Furthermore $I(0) = 0$, $I(s^*/2) < 0$ for large enough $n$ as $k/n \to 0$ and $s^*$ is a fixed value. From this condition we can deduce at least two points in $(0, s^*/2)$ have a horizontal tangential line. Pick the smallest and the largest among them from $(0, s^*/2)$ and call them $s_1$ and $s_2$ respectively.\\
        Taylor expansion of $I$ at $x=0$ becomes
        \begin{equation*}
            I(x) \sim \frac{\lambda k}{n}x - 2(\beta-1)x^2 = -2(\beta-1)\Big(x - \frac{\lambda k}{4(\beta - 1)n}\Big)^2 + O\Big(\frac{k^2}{n^2}\Big).
        \end{equation*}
        and $I$ at $x = s^*/2$ becomes
        \begin{equation*}
            \begin{split}
            I(x) &\sim I(s^*/2) + \frac{\lambda k}{n}\Big(x - \frac{s^*}{2}\Big) + \frac{I''(s^*/2)}{2}\Big(x - \frac{s^*}{2}\Big)^2 \\
            &= I(s^*/2) + \frac{I''(s^*/2)}{2}\Big(x - \frac{s^*}{2} + \frac{\lambda k}{n I''(s^*/2)}\Big)^2 - O\Big(\frac{k^2}{n^2}\Big).
            \end{split}
        \end{equation*}
        Substitute $I''(s^*/2)$ in terms of $s^*$ and $\beta$ then we have
        \begin{equation*}
            s_1 \sim \frac{\lambda k}{4(\beta - 1)n}, \qquad \textrm{and} \qquad s_2 \sim \frac{s^*}{2} - \frac{\lambda k}{nI''(s^*/2)}.
        \end{equation*}
    \end{proof}

    \begin{lemma} \label{prop: C^3 lemma}
    Suppose a function $f : \mathbb{R} \to \mathbb{R} \in \mathcal{C}^3$ and a real number $s_0$ satisfies
        \begin{enumerate}[leftmargin=*]
            \item $f'''(t) \geq 0$ for all $t \geq s_0$.
            \item $f'(t) \leq 0$ for all $t \geq s_0$.
            \item $f'(s_0) = 0$.
        \end{enumerate}
    Then for any $s \leq x \leq y$, $f(y) - f(x) \leq \frac{f'(y)}{2}(y-x)$ holds.
    \end{lemma}
    
    \begin{proof}
        Without loss of generality, set $s = 0$.
        \begin{equation} \label{eq: 8-1}
            \begin{split} 
                f(x) - f(y) + \frac{(y-x)f'(y)}{2} &= -\int_x^y f'(t)dt + \int_x^y \frac{f'(y)}{2}dt \\
                &= \int_x^y t\Big(\frac{f'(t)}{t} - \frac{f'(y)}{y}\Big) dt + f'(y)\frac{x^2 - xy}{2y}.
            \end{split}
        \end{equation}
        Rightmost term of \eqref{eq: 8-1} is non-negative. Furthermore, $tf''(t) - f'(t) = \int_0^t f''(t)ds - \int_0^t f''(s)ds = \int_0^t (f''(t) - f''(s)) ds \geq 0$ ensures that
        \begin{equation*}
            \frac{d}{dt}\Big(\frac{f'(t)}{t}\Big) = \frac{tf''(t) - f'(t)}{t^2} \geq 0,
        \end{equation*}
        which finishes the proof.
    \end{proof}
    
    \begin{remark*}
    The function $I(x)$ defined in Lemma \ref{prop: auxiliary function f and I} satisfies the all the conditions of Lemma \ref{prop: C^3 lemma} with $s = s_1$.
    \end{remark*}

    \begin{proof} (Lemma \ref{lem: hitting time estimate from below}) \\
    Suppose $n$ is even, and consider a new sequence $M_i := nS^+(Y_i)/2$. This is not a Markov chain, but $M_i$ forms an integral sequence in $[0, n/2]$ whose difference between any consecutive terms are in $\{-1, 0, 1\}$. Define $\tau_c := \min \{t \geq 0 : M_i = c\}$ and
    \begin{equation*}
        \begin{split}
        B_c &:= \sup\mathbb{E}_c[\tau_{c+1}] \\
        p_c &:= \sup\mathbb{P}_c[M_{i+1} - M_i = 1 | \mathcal{F}_i] \\
        q_c &:= \sup\mathbb{P}_c[M_{i+1} - M_i = -1 | \mathcal{F}_i] \\
        r_c &:= \sup\mathbb{P}_c[M_{i+1} - M_i = 0 | \mathcal{F}_i].
        \end{split}
    \end{equation*}
    All these four terms depend on $i' := i(\textrm{mod }k)$ and the previous $i - i'$ update history. Therefore, we take supremum over all possible $i'$s and $\mathcal{N}_i$. We can setup an estimate for $B_c$. We know that $B_0 = O(1)$, so for $1 \leq c \leq ns^* + \alpha \sqrt{n}$, for any time $i$
    \begin{equation} \label{eq: 8-2}
        \begin{split}
            \mathbb{E}_c[\tau_{c+1}|\mathcal{F}_i] &\leq p_c + q_c(B_{c-1} + B_c + 1) + r_c(B_c + 1) \\
                &\qquad \longrightarrow (1-q_c-r_c)B_c \leq (p_c + q_c + r_c) + q_cB_{c-1}
        \end{split}
    \end{equation}
    Regardless of $i (\textrm{mod } k)$, we know that
    \begin{equation*}
        \begin{split}
        p_c &= \Big[\frac{1-s}{2} + O\Big(\frac{k}{n}\Big)\Big]p_{+}\Big(s + \frac{1}{n}\Big) = \Big[\frac{n-2c}{2n} + O\Big(\frac{k}{n}\Big)\Big]p_{+}\Big(\frac{1}{n}(2c+1)\Big) \\
        q_c &= \Big[\frac{1+s}{2} + O\Big(\frac{k}{n}\Big)\Big]p_{-}\Big(s - \frac{1}{n}\Big) = \Big[\frac{n+2c}{2n} + O\Big(\frac{k}{n}\Big)\Big]p_{-}\Big(\frac{1}{n}(2c-1)\Big),
        \end{split}
    \end{equation*}
    hence Equation \eqref{eq: 8-2} becomes
    \begin{equation} \label{eq: 8-3}
        \begin{split}
        B_c &\leq A + \Big[\frac{n+2c}{n-2c} + O\Big(\frac{k}{n}\Big)\Big)\Big]\frac{p_{-}\Big(\frac{1}{n}(2c-1)\Big)}{p_{+}\Big(\frac{1}{n}(2c+1)\Big)}B_{c-1} \\
        &\leq A + \Big[\frac{n+2c}{n-2c} + O\Big(\frac{k}{n}\Big)\Big)\Big]\Big(1 + O\Big(\frac{1}{n}\Big)\Big)e^{-\frac{2\beta}{n}(2c+1)}B_{c-1} \\
        &\leq A + \Big[\frac{n+2c}{n-2c} + \frac{\lambda k}{n}\Big]e^{-\frac{2\beta}{n}(2c+1)}B_{c-1}
        \end{split}
    \end{equation}
    for some positive constant $A$ and $\lambda$. 
    We know $B_1 = O(1)$, so repeating Equation \eqref{eq: 8-3} gives
    \begin{equation} \label{eq: 8-4}
        \begin{split}
        B_c \lesssim \sum_{j = 1}^c e^{-\frac{2\beta}{n}(c^2 - j^2)}\frac{n+2c}{n-2c}...\frac{n+2j}{n-2j}e^{(c - j)\frac{\lambda k}{n}}.
        \end{split}
    \end{equation}
    For any $1\leq j \leq c \leq ns^* + \alpha\sqrt{n}$,
    \begin{equation*}
        \begin{split}
        \frac{n + 2c}{n - 2c} ... \frac{n + 2j}{n - 2j} &= 
        \frac{(n/2 + c)!/(n/2 + j - 1)!}{(n/2 - j)!/(n/2 - c - 1)!} = \frac{(n/2 + c)!(n/2 - c - 1)!}{(n/2 - j)!(n/2 + j-1)!} \\
        &\simeq \frac{(n/2 + c)!(n/2 - c)!}{(n/2 - j)!(n/2 + j)!} \simeq \frac{(n/2 + c)^{n/2 + c}(n/2 - c)^{n/2 - c}}{(n/2 - j)^{n/2 - j}(n/2 + j)^{n/2 + j}} \\
        &= \exp\Big(nf(c/n) - nf(j/n)\Big),
        \end{split}
    \end{equation*}
    where $f(x) = \left(\frac{1}{2} + x\right)\log\left(\frac{1}{2} + x\right) + \left(\frac{1}{2} - x\right)\log\left(\frac{1}{2} - x\right) + \log 2$. Hence \eqref{eq: 8-4} becomes
    \begin{equation*}
        \begin{split}
        B_c &\lesssim \sum_{j=1}^c \exp\Big\{-\frac{2\beta}{n}(c^2 - j^2) + n\big[f\big(\frac{c}{n}\big) - f\big(\frac{j}{n}\big)\big] + (c-j)\frac{\lambda k}{n}\Big\} \\
        &= \sum_{j=1}^c \exp\Big\{nI\big(\frac{c}{n}\big) - nI\big(\frac{j}{n}\big)\Big\},
        \end{split}
    \end{equation*}
    where $I$ is from Proposition \ref{prop: auxiliary function f and I}. Eventually we need to bound $\sum B_c$, which is
    \begin{equation*}
        \begin{split}
        \sum_{i = 1}^{(ns^* + \alpha\sqrt{n})/2} B_i &\leq \sum_{i = 1}^{(ns^* + \alpha\sqrt{n})/2}\sum_{j = 1}^{i} \exp\Big\{nI\big(\frac{i}{n}\big) - nI\big(\frac{j}{n}\big)\Big\} \\
        &\simeq n^2 \int_0^{\frac{s^*}{2} + \frac{\alpha}{2\sqrt{n}}}\int_0^y \exp[nI(y) - nI(x)] dx dy.
        \end{split}
    \end{equation*}
    The problem changed to calculate the integral over triangular area $\{0 \leq y \leq s^*/2 + \alpha/2\sqrt{n}, 0 \leq x \leq y\}$. We split the domain into by 8 pieces(Figure \ref{fig: domain_division_1}).
    
    \begin{figure}[h]
        \includegraphics[scale = 0.4]{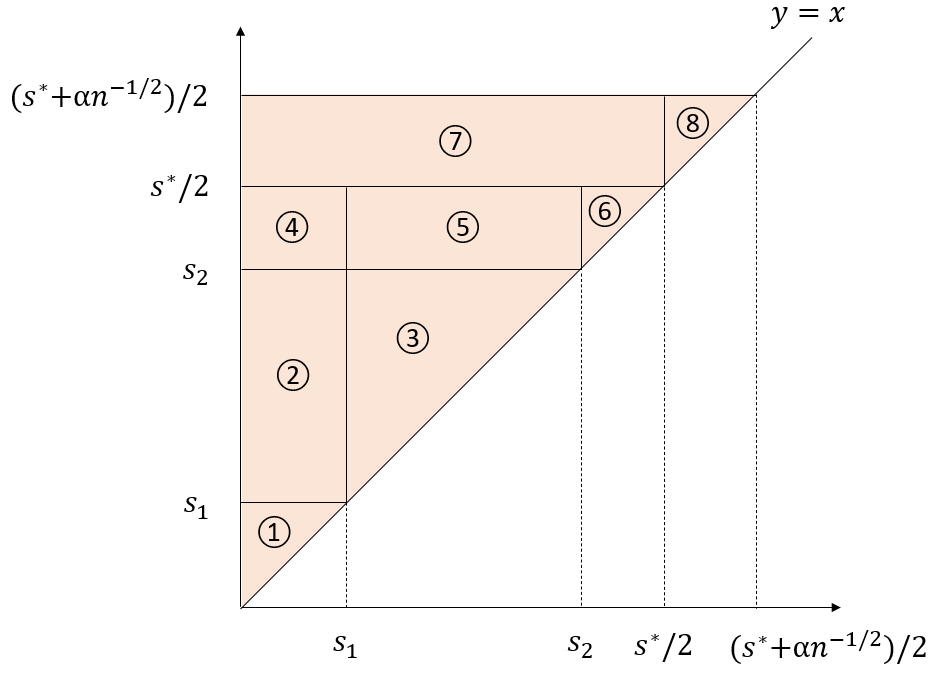}
        \centering
        \caption{A graphical domain division by 8 pieces.}
        \label{fig: domain_division_1}
    \end{figure}
    
    \medskip
    
    \noindent \textbf{Piece 1}: $\{0 \leq y \leq s_1, 0 \leq x \leq y\}$ \\
    For any $x \leq y$, we have
    \begin{equation*}
        |I(y) - I(x)| \leq (y-x)\sup_{x<s<y} I'(s) \leq (y-x)\sup_{0<s<s_1}I'(s) = (y-x)O\Big(\frac{k}{n}\Big) = O\Big(\frac{k^2}{n^2}\Big),
    \end{equation*}
    therefore the integration over this domain becomes $n^2 s_1^2 \exp\left(O\left(\frac{k^2}{n}\right)\right) = O(k^2)$.
    \medskip
    
    \noindent \textbf{Piece 2}: $\{s_1 \leq y \leq s_2, 0 \leq x \leq s_1\}$ \\
    \begin{equation*}
        \begin{split}
            &n^2\int_0^{s_1} \exp[-nI(x)] dx \int_{s_1}^{s_2} \exp[nI(y)] dy \\
            &\qquad \qquad \leq n^2s_1 \exp[-nI(0)] \int_{s_1}^{s_2} \exp[nI(y)]dy \\
            &\qquad \qquad \leq n^2s_1 \exp[nI(s_1)-nI(0)] \int_{s_1}^{s_2} \exp[nI(y)-nI(s_1)]dy \\
            &\qquad \qquad = O(kn) \int_{s_1}^{s_2} \exp[nI(y)-nI(s_1)]dy.
        \end{split}
    \end{equation*}
    Around $y = s_1$ we have $nI(y) - nI(s_1) \simeq -2n(\beta-1)(y-s_1)^2$. If $y$ is far from $s_1$, since $I$ is a decreasing function on $(s_1, s_2)$, $nI(y) - nI(s_1)$ rapidly becomes small. Splitting the domain of integration into $(s_1, s_1 + \log n / n)$ and $(s_1 + \log n / n, s_2)$ gives
    \begin{equation*}
        O(kn)\int_{s_1}^{s_2} \exp[nI(y)-nI(s_1)] dy = O(kn)O\left(\frac{\log n}{\sqrt{n}}\right) = O(k\sqrt{n}\log n).
    \end{equation*}
    \medskip
    
    \noindent \textbf{Piece 3}: $\{s_1 \leq y \leq s_2, s_1 \leq x \leq y\}$ \\
    This domain should be divided into 6 pieces again. This continues at the end of the proof.
    \medskip
    
    \noindent \textbf{Piece 4}: $\{s_2 \leq y \leq s^*/2, 0 \leq x \leq s_1\}$ \\
    Since $I$ is increasing on both $(0, s_1)$ and $(s_2, s^*/2)$,
    \begin{equation*}
        n^2 \int_0^{s_1}\exp[-nI(x)]dx \int_{s_2}^{s^*/2}\exp[nI(y)]dy \leq n^2 s_1 \exp[-nI(0)]\Big(\frac{s^*}{2} - s_2\Big)\exp[nI(s^*/2)] = O(k^2).
    \end{equation*}
    \medskip
    
    \noindent \textbf{Piece 5}: $\{s_2 \leq y \leq s^*/2, s_1 \leq x \leq s_2\}$ \\
    Similar to Piece 2.
    \begin{equation*}
        \begin{split}
        &n^2 \int_{s_1}^{s_2}\exp[-nI(x)]dx \int_{s_2}^{s^*/2}\exp[nI(y)]dy \\
        &\qquad \qquad \leq n^2 \int_{s_1}^{s_2}\exp[-nI(x)]dx \left(\frac{s^*}{2} - s_2\right)\exp[nI(s^*/2)]\\
        &\qquad \qquad = n^2 \int_{s_1}^{s_2}\exp[nI(s_2)-nI(x)]dx \left(\frac{s^*}{2} - s_2\right) \exp\left[ nI(s^*/2) - nI(s_2)\right] \\
        &\qquad \qquad \lesssim n^2 \int_{s_1}^{s_2}\exp[nI(s_2)-nI(x)]dx \left(\frac{s^*}{2} - s_2\right) \\
        &\qquad \qquad = n^2 O(\frac{\log n}{\sqrt{n}})O\Big(\frac{k}{n}\Big) = O(k\sqrt{n}\log n).
        \end{split}
    \end{equation*}
    \medskip
    
    \noindent \textbf{Piece 6}: $\{s_2 \leq y \leq s^*/2, s_2 \leq x \leq s^*/2\}$ \\
    Similar to Piece 1.
    \begin{equation*}
        n^2 \iint \exp[nI(y)-nI(x)]dydx \lesssim n^2\Big(\frac{s^*}{2} - s_2\Big)^2 = O(k^2).
    \end{equation*}
    \medskip
    
    \noindent \textbf{Piece 7}: $\{s^*/2 \leq y \leq s^*/2 + \alpha/2\sqrt{n}, 0 \leq x \leq s^*/2\}$ \\
    Similar to Piece 4, 5, 6. The integral is
    \begin{equation*}
        n^2 \int_{s^*/2}^{s^*/2 + \alpha/2\sqrt{n}}\exp[nI(y)]dy\int_0^{s^*/2}\exp[-nI(x)]dx.
    \end{equation*}
    We have $I(\frac{s^*}{2}) < 0$. Taylor expansion at $\frac{s^*}{2}$ for large enough $n$ ensures $I\left(\frac{s^*}{2}+ \frac{\alpha}{2\sqrt{n}}\right) < 0$. Therefore
    \begin{equation*}
        n^2 \int_{s^*/2}^{s^*/2 + \alpha/2\sqrt{n}}\exp[nI(y)]dy \lesssim n^{3/2}.
    \end{equation*}
    Now, split the range of $x$ by $[0, s_1] \cup [s_1, s_2] \cup [s_2, s^*/2]$. Then each integral with respect to $x$ gives $O(k/n), O(\log n/\sqrt{n})$ and $ O(k/n)$. In conclusion,
    \begin{equation*}
        n^2 \iint \exp[nI(y)-nI(x)]dydx = O(k\sqrt{n} + n\log n) = O(n\log n).
    \end{equation*}

    \noindent \textbf{Piece 8}: $\{s^*/2 \leq y \leq s^*/2 + \alpha/2\sqrt{n}, s^*/2 \leq x \leq s^*/2 + \alpha/2\sqrt{n}\}$ \\
    Similar to Piece 1.
    \begin{equation*}
        n^2 \iint \exp[nI(y)-nI(x)]dydx \lesssim n^2\Big(\frac{\alpha}{2\sqrt{n}}\Big)^2 = O(n).
    \end{equation*}
    \medskip
    
    \noindent We divide \textbf{Piece 3}: $\{s_1 \leq y \leq s_2, s_1 \leq x \leq y\}$ again into 6 pieces with a parameter $0 < \epsilon = o(1)$. $\epsilon$ will be chosen after the calculation(Figure \ref{fig: domain_division_2}).

    \begin{figure}[h]
        \includegraphics[scale = 0.4]{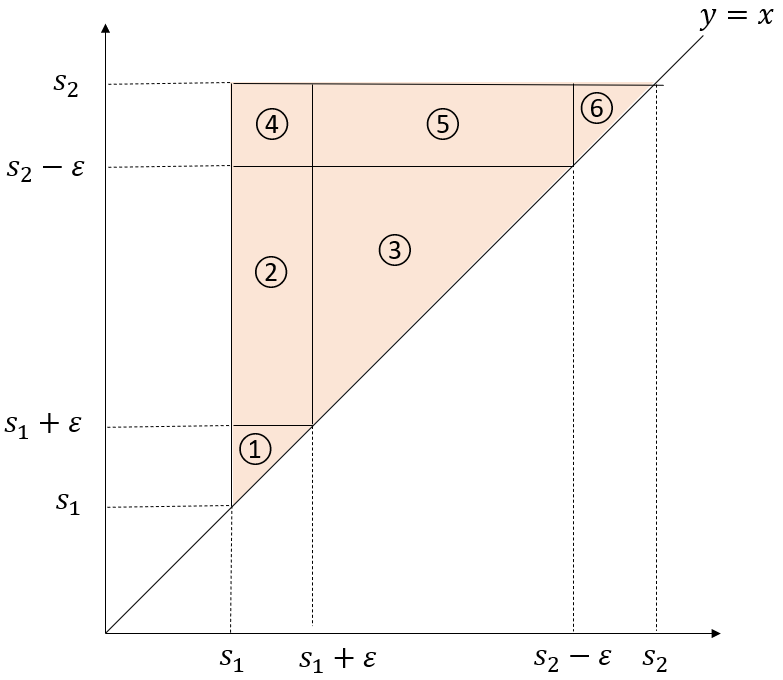}
        \centering
        \caption{A graphical domain division for Piece 3 from Figure \ref{fig: domain_division_1}, by 6 pieces.}
        \label{fig: domain_division_2}
    \end{figure}
    
    \noindent \textbf{Piece 3-1} : $\{s_1 \leq y \leq s_1 + \epsilon, s_1 \leq x \leq y \}$ \\
    For small enough $\epsilon>0$, $s_1 \leq x \leq y \leq s_1 + \epsilon$ implies $I(x) \geq I(y)$. Therefore,
    \begin{equation*}
        n^2 \int_{s_1}^{s_1 + \epsilon} \int_{s_1}^y \exp[nI(y)-nI(x)]dxdy \lesssim n^2\epsilon^2.
    \end{equation*}
    
    \noindent \textbf{Piece 3-2} : $\{s_1 + \epsilon \leq y \leq s_2 - \epsilon, s_1 \leq x \leq s_1 + \epsilon \}$ \\
    \begin{equation*}
        n^2\int_{s_1 + \epsilon}^{s_2 - \epsilon}\int_{s_1}^{s_1 + \epsilon} \exp[nI(y)-nI(x)] dx dy \leq n^2 \epsilon \int_{s_1 + \epsilon}^{s_2 - \epsilon}\exp[nI(y) - nI(s_1 + \epsilon)] dy
    \end{equation*}
    Near $y = s_1 + \epsilon$, the approximation $nI(y) - nI(s_1 + \epsilon) \simeq -2n(\beta-1)\{(y-s_1)^2 - \epsilon^2\}$ holds. $I(y)$ is uniformly bounded above by $I(s_1 + \epsilon)$ when $y$ is far from $s_1 + \epsilon$. Hence the above is bounded by $O(n^2\epsilon)$, provided that $n\epsilon \to \infty$.

    \noindent \textbf{Piece 3-3} : $\{s_1 + \epsilon \leq y \leq s_2 - \epsilon, s_1 + \epsilon \leq x \leq y \}$ \\
    Proposition \ref{prop: C^3 lemma} can be applied to the function $I$ in this domain.
    \begin{equation*}
        \begin{split}
            &n^2\int_{s_1 + \epsilon}^{s_2 - \epsilon}\int_{s_1 + \epsilon}^{y} \exp[nI(y)-nI(x)] dx dy \\
            &\qquad \qquad \leq n^2\int_{s_1 + \epsilon}^{s_2 - \epsilon}\int_{s_1 + \epsilon}^{y} \exp[nI'(y)(y-x)/2] dx dy \\
            &\qquad \qquad = n^2 \int_{s_1 + \epsilon}^{s_2 - \epsilon} -\frac{2}{nI'(y)}\Big\{1 - \exp\Big[nI'(y)\big(\frac{y}{2} - \frac{s_1 + \epsilon}{2}\big)\Big]\Big\} dy \\
            &\qquad \qquad \leq n^2 \int_{s_1 + \epsilon}^{s_2 - \epsilon} -\frac{2}{nI'(y)}dy
        \end{split}
    \end{equation*}
    Near $y = s_1 + \epsilon$ we have $I'(y) \simeq 4(\beta - 1)(y-s_1)$, while near $y = s_2 - \epsilon$ we have $I'(y) \simeq O(1)(s_2 - y)$. For $y$ in between $s_1 + \epsilon$ and $s_2 - \epsilon$, $|I'(y)|$ is bounded below by a positive constant. Therefore,
    \begin{equation*}
        n^2 \int_{s_1 + \epsilon}^{s_2 - \epsilon} -\frac{2}{nI'(y)}dy \lesssim n O(\log(1/\epsilon)) + O(n) =O(-n\log\epsilon) + O(n).
    \end{equation*}

    \noindent \textbf{Piece 3-4} : $\{s_2 - \epsilon \leq y \leq s_2, s_1 \leq x \leq s_1 + \epsilon \}$ \\
    Similar to Piece 3-1. In this domain we have $I(x) \geq I(y)$ for $x \leq y$. Therefore
    \begin{equation*}
        n^2 \iint \exp[nI(y)-nI(x)]dydx \lesssim n^2\epsilon^2.
    \end{equation*}

    \noindent \textbf{Piece 3-5} : $\{s_2 - \epsilon \leq y \leq s_2, s_1 + \epsilon \leq x \leq s_2 - \epsilon \}$ \\
    Similar to Piece 3-2.
    \begin{equation*}
            n^2\int_{s_2 - \epsilon}^{s_2}\int_{s_1 + \epsilon}^{s_2 - \epsilon} \exp[nI(y)-nI(x)] dx dy = n^2 \epsilon \int_{s_1 + \epsilon}^{s_2 - \epsilon} \exp[nI(s_2 - \epsilon)-nI(x)] dx.
    \end{equation*}
    Near $x = s_2 - \epsilon$, the approximation $nI(s_2 - \epsilon) - nI(x) \simeq O(n)\epsilon(x - s_2 + \epsilon)$ holds. Hence the above is bounded by $O(n^2\epsilon)$ in an analogous way.
    
    \noindent \textbf{Piece 3-6} : $\{s_2 - \epsilon \leq y \leq s_2, s_2 - \epsilon \leq x \leq y \}$ \\
    Similar to Piece 3-1, $I(x) \geq I(y)$ for $x \leq y$ implies
     \begin{equation*}
        n^2 \iint \exp[nI(y)-nI(x)]dydx \lesssim n^2\epsilon^2.
    \end{equation*}
    
    Now set $\epsilon = O(\log n / n)$. Under $k = o(\sqrt{n})$ assumption, all the integrals over the domains are $O(n\log n)$, which ensures that
    \begin{equation*}
        \mathbb{E}_0\Big[\min \big\{i \geq 0 : S^+(Y_i) \geq s^* + \frac{\alpha}{\sqrt{n}}\big\}\Big] = O(n\log n).
    \end{equation*}
    As this is true for any $\alpha > 0$, after you reach the stopping time
    \begin{equation*}
        \tau_{**} := \Big\{i \geq 0 : S^+(Y_i) \geq s^* + \frac{\alpha}{\sqrt{n}}\Big\},
    \end{equation*}
    Consider the time $\lceil \tau_{**}/k \rceil$. We have 
    \begin{equation*}
        S^+(X_{\lceil \tau_{**}/k \rceil}) = S^+(Y_{k\lceil \tau_{**}/k \rceil}) \geq s^* + \frac{\alpha}{\sqrt{n}} - \frac{2k}{n}
    \end{equation*}
    From the condition $k = o(\sqrt{n})$, adjusting $\alpha$ gives
    \begin{equation*}
        \lceil \frac{\tau_{**}}{k}\rceil \sim \tau_* := \min \big\{t \geq 0 : S^T_t \geq s^* + \frac{\alpha}{\sqrt{n}}\big\}.
    \end{equation*}
    Therefore, $\mathbb{E}_0[\tau_*] = O(n\log n / k)$.
    \end{proof}
    \medskip

    \bibliographystyle{alpha}
    \bibliography{cutoffbl}

\end{document}